\documentclass[11pt]{amsart}
\usepackage{amssymb,latexsym,amsfonts}
\usepackage{amsmath}
\usepackage{amscd}
\usepackage{pb-diagram,pb-xy}
\usepackage{tikz}
\usetikzlibrary{shapes,arrows}
\usepackage{graphicx}
\usepackage{hyperref}
\usepackage{color}
\usepackage[matrix,arrow]{xy}

\usepackage{fullpage}
\newcommand{\str}{{\star}}

\newcommand{\cc}{{\mathbf {c}} }

\newcommand{\Scal}{{\rm Scal}}

\newcommand{\Ric}{{\rm Ric}}

\DeclareMathAlphabet{\mathpzc}{OT1}{pzc}{m}{it}

\newtheorem{theorem}{Theorem}[section]
\newtheorem{corollary}[theorem]{Corollary}

\newtheorem{lemma}[theorem]{Lemma}

\newtheorem{proposition}[theorem]{Proposition}
\newtheorem*{thm-a}{Theorem\! A}
\newtheorem*{thm-aa}{Theorem\! ${\rm A}^{\prime}$}
\newtheorem*{cor-bb}{Corollary\! ${\rm B}^{\prime}$}
\newtheorem*{thm-cc}{Theorem\! ${\rm C}^{\prime}$}
\newtheorem*{cor-a}{Corollary\! A}
\newtheorem*{thm-b}{Theorem\! B}
\newtheorem*{cor-b}{Corollary\! B}
\newtheorem*{thm-c}{Theorem\! C}
\newtheorem*{cor-c}{Corollary\! C}
\newtheorem*{thm-d}{Theorem\! D}
\newtheorem*{thm-dd}{Theorem\! ${\rm D}^{\prime}$}
\newtheorem*{cor-dd}{Corollary\! ${\rm D}^{\prime}$}
\newtheorem*{cor-d}{Corollary\! D}
\newtheorem*{thm-e}{Theorem\! E}
\newtheorem*{cor-e}{Corollary\! E}
\newtheorem*{thm-f}{Theorem\! F}
\newtheorem*{cor-f}{Corollary\! F}

\newtheorem*{conj-d}{Conjecture D}
\newtheorem*{conj-c}{Conjecture C}
\newtheorem{thm}{Theorem}[section]
\theoremstyle{definition}
\newtheorem{definition}[thm]{Definition}

\newtheorem*{remark}{Remark}
\newtheorem*{remarks}{Remarks}

\newtheorem{examples}{Examples}
\newtheorem{example}{Example}[section]

\begin{document}
\author{Mohammed Larbi Labbi} 
\address{Department of Mathematics\\
 College of Science\\
  University of Bahrain\\
  32038, Bahrain.}
\email{mlabbi@uob.edu.bh}
\renewcommand{\subjclassname}{
  \textup{2020} Mathematics Subject Classification}

\title{On some critical Riemannian metrics and Thorpe-type conditions}  \maketitle
\begin{abstract}
We study critical metrics of higher-order curvature functionals on compact Riemannian $n$-manifolds $(M,g)$. For an integer $k$ with $2 \leq 2k \leq n$, let $R^k$ denote the $k$-th exterior power of the Riemann curvature tensor, viewed as a double form ($2k$-th Thorpe tensor). We investigate the Riemannian functionals
\[
H_{2k}(g)=\int_M \operatorname{tr}(R^k)\,\mathrm{dvol}_g
\quad\text{and}\quad
G_{2k}(g)=\int_M \|R^k\|^2\,\mathrm{dvol}_g,
\]
which generalize the Hilbert--Einstein functional  and the total squared norm curvature, obtained for $k=1$ respectively.

Using the formalism of double forms, we develop a systematic variational framework yielding compact first variation formulas for these functionals. Two key lemmas streamline the variational computations. A central technical ingredient is a generalization of the classical Lanczos identity to symmetric double forms of arbitrary even degree, providing explicit algebraic relations between the tensors $\cc^{2k-1}(R^k \circ R^k)$ and $\cc^{4k-1}(R^{2k})$.

As a main geometric application, we introduce $(2k)$-Thorpe and $(2k)$-anti-Thorpe metrics, defined by self-duality and anti-self-duality conditions on $g^{r-2k}R^k$ in even dimensions $n=2r$. In the critical dimension $n=4k$, these metrics are absolute minimizers of $G_{2k}$, with the minimum determined by the Euler characteristic. For $n>4k$, they satisfy a harmonicity property leading to rigidity results under suitable curvature positivity assumptions.

We further establish equivalences among variational criticality conditions. For hyper-$(2k)$-Einstein metrics, characterized by $\cc R^k=\lambda g^{2k-1}$, being critical for $G_{2k}$ is equivalent to being $(4k)$-Einstein and to being weakly $(2k)$-Einstein. In the locally conformally flat setting, we classify all $4$-Thorpe metrics, showing that they are either space forms or Riemannian products $\mathbb{S}^r(c) \times \mathbb{H}^r(-c)$.

Overall, the paper explores the intermediate range $2 \leq 2k \leq n/2$, extending classical results on Einstein and Thorpe metrics and revealing new connections between curvature duality, harmonicity, and variational properties of higher-order curvature tensors.
\end{abstract}


\tableofcontents

\section{Introduction and Summary of the Main Results} \label{intro}

The study of critical points of curvature functionals on the space of Riemannian metrics lies at the heart of modern differential geometry, with deep links to geometric analysis, general relativity, and theoretical physics. Classical examples include Einstein metrics---critical points of the total scalar curvature---and metrics critical for the \(L^2\)-norm of the full curvature tensor. In this paper we investigate natural higher-order generalizations of these functionals, focusing on the \(L^2\)-norm of the exterior powers of the curvature tensor, viewed through the convenient formalism of double forms.

Throughout, \((M,g)\) denotes a smooth, compact Riemannian \(n\)-manifold without boundary. We regard the Riemann curvature tensor \(R\) as a symmetric \((2,2)\) double form. For an integer \(k\) with \(2 \leq 2k \leq n\), let \(R^k\) denote the \(k\)-th exterior power of \(R\) in the exterior algebra of double forms; this \((2k,2k)\) double form is also known as the \emph{\(2k\)-th Thorpe tensor} or the Gauss--Kronecker tensor. We study the functional
\[
G_{2k}(g)=\int_M\|R^k\|^2\,\mathrm{dvol}_g=\int_M\langle R^k,R^k\rangle\,\mathrm{dvol}_g,
\]
which for \(k=1\) coincides with the classical \(L^2\)-curvature functional. Our main goal is to characterize the critical points of \(G_{2k}\) and to uncover their relations with several natural classes of Riemannian metrics, including generalized Einstein conditions and metrics satisfying duality properties.

\subsection{\((2k)\)-Thorpe and \((2k)\)-Anti-Thorpe Riemannian Manifolds}

Assume \(n=2r\) is even and \(2\leq 2k\leq r\). We say that \((M,g)\) is \emph{\((2k)\)-Thorpe} if
\[
\star\!\bigl(g^{r-2k}R^k\bigr)=g^{r-2k}R^k,
\]
and \emph{\((2k)\)-anti-Thorpe} if
\[
\star\!\bigl(g^{r-2k}R^k\bigr)=-g^{r-2k}R^k,
\]
where \(\star\) is the (double) Hodge star operator on double forms (see §\ref{Hodge}).  
For \(k=1\) one recovers Einstein metrics (2-Thorpe) and conformally flat metrics with zero scalar curvature (2-anti-Thorpe) \cite{Labbi-double-forms}. When \(r=2k\) (i.e.\ \(n=4k\)) these reduce to the classical Thorpe/anti-Thorpe manifolds studied in \cite{Thorpe,Kim,Labbi-JAUMS,Labbi-double-forms}. The present paper systematically investigates the intermediate range \(2\leq 2k\leq r\), which has not been treated in full generality before.

One of our first results shows that in the critical dimension \(n=4k\), Thorpe and anti-Thorpe metrics not only are critical for \(G_{2k}\) but actually minimize it under appropriate topological conditions.

\begin{proposition}\label{prop-e}
Let \((M,g)\) be a compact Riemannian manifold of dimension \(n=4k\). Then:
\begin{enumerate}
    \item If \(\chi(M)>0\), the functional \(G_{2k}\) attains its absolute minimum exactly on the \(2k\)-Thorpe metrics, and the minimum value equals \((2\pi)^{2k}(2k)!\,\chi(M)\).
    \item If \(\chi(M)<0\), the functional \(G_{2k}\) attains its absolute minimum exactly on the \(2k\)-anti-Thorpe metrics, and the minimum value equals \(-(2\pi)^{2k}(2k)!\,\chi(M)\).
\end{enumerate}
\end{proposition}

For dimensions \(n\geq 4k\) we discover a remarkable harmonicity property.

\begin{proposition}\label{prop-f}
Let \((M,g)\) be a Riemannian manifold of even dimension \(n=2r\) and let \(2\leq 2k\leq r\). If \((M,g)\) is \(2k\)-Thorpe (resp.\ \(2k\)-anti-Thorpe), then the tensor \(R^k\) is a harmonic double form.
\end{proposition}

This harmonicity leads to rigidity results under positivity conditions on the curvature. Recall that \(R\) is called \emph{\(r\)-positive} if the sum of the lowest \(r\) eigenvalues of the curvature operator is positive.

\begin{corollary}\label{corollary-f}
Let \((M,g)\) be a closed, connected \((2k)\)-Thorpe manifold of dimension \(n\geq 4k\geq 4\). Then:
\begin{itemize}
    \item[(a)] If \(R\) is \(\bigl\lfloor\frac{n-2k+1}{2}\bigr\rfloor\)-positive, then \(R^k\) has constant sectional curvature.
    \item[(b)] If \(R\) is \(\bigl\lfloor\frac{n-2k+2}{2}\bigr\rfloor\)-positive, then \((M,g)\) is \emph{hyper \((2k)\)-Einstein} (i.e., the first contraction \(\mathbf{c}R^k\) has constant sectional curvature).
\end{itemize}
\end{corollary}

\begin{corollary}\label{corollary-g}
Let \((M,g)\) be a closed, connected \((2k)\)-anti-Thorpe manifold of dimension \(n\geq 4k\geq 4\). If \(R\) is \(\bigl\lfloor\frac{n-2k+2}{2}\bigr\rfloor\)-positive, then \(R^k=0\) (i.e., \((M,g)\) is \emph{\(k\)-flat}).
\end{corollary}

We also obtain obstructions and classification results for hypersurfaces in Euclidean space and for conformally flat manifolds.

\begin{proposition}\label{prop-h}
For \(n\geq 4k\) there are no compact embedded hypersurfaces in \(\mathbb{R}^{n+1}\) that are \(2k\)-anti-Thorpe. Moreover, the round sphere is the only compact embedded \(2k\)-Thorpe hypersurface in \(\mathbb{R}^{n+1}\).
\end{proposition}

\begin{proposition}\label{prop-i}
Let \((M,g)\) be a locally conformally flat \(n\)-manifold with \(n=2r\geq 8\). Then \((M,g)\) is \(4\)-Thorpe if and only if it is either flat or covered by one of the following: the standard sphere \(\mathbb{S}^n\), hyperbolic space \(\mathbb{H}^n\), or the product \(\mathbb{S}^r(c)\times\mathbb{H}^r(-c)\).
\end{proposition}

Furthermore, we show that the Riemannian product \(\mathbb{S}^r(c)\times\mathbb{H}^r(-c)\) is \(4\)-Thorpe but not \(4\)-anti-Thorpe. However, it is  \(6\)-anti-Thorpe for all \(r\geq 6\).

\subsection{Weakly \((2k)\)-Einstein Metrics}

When \(n>4k\), \((2k)\)-Thorpe and \((2k)\)-anti-Thorpe metrics are not automatically critical for \(G_{2k}\). Instead, we characterize when they become critical.

\begin{proposition}\label{prop-j}
Let \(n\) be even and \(n>4k\). A \((2k)\)-Thorpe (resp.\ \((2k)\)-anti-Thorpe) metric is a critical point of \(G_{2k}\) restricted to unit volume metrics if and only if
\begin{equation}\label{Weak-Einstein}
\frac{1}{(2k-1)!}\,\mathbf{c}^{2k-1}(R^k\circ R^k)=\frac{2k}{n}\|R^k\|^2\,g.
\end{equation}
\end{proposition}

Following \cite{EPS}, we call metrics satisfying \eqref{Weak-Einstein} \emph{weakly \((2k)\)-Einstein}. For \(k=1\) this reduces to the notion of weakly Einstein metrics introduced by Euh, Park, and Sekigawa.

\subsection{\((2k)\)-Einstein Metrics}

Another natural class of critical metrics arises from the Hilbert--Lovelock functionals
\[
H_{2k}(g)=\int_M h_{2k}(g)\,\mathrm{dvol}_g,
\]
where \(h_{2k}(g)\) is the \(2k\)-th Gauss--Bonnet curvature (see Definition~\ref{GB-definition}). Critical points of \(H_{2k}\) restricted to unit volume metrics are called \emph{\((2k)\)-Einstein metrics}; they satisfy
\[
\frac{\mathbf{c}^{2k-1}R^k}{(2k-1)!}=\frac{2k}{n}h_{2k}\,g.
\]
For \(k=1\) this is the classical Einstein equation. We prove that every \((2k)\)-Thorpe metric is \((2k)\)-Einstein, and every \((2k)\)-Einstein metric has constant \(h_{2k}\). Moreover, we establish two further properties:

\begin{proposition}\label{prop-d}
In dimensions \(n>4\), a Riemannian metric that is both \(2\)-Einstein and \(4\)-Einstein is a critical point of
\[
G_2=\int_M\|R\|^2\,\mathrm{dvol}_g
\]
(restricted to unit volume metrics).
\end{proposition}

\begin{proposition}\label{prop-z}
Let \((M,g)\) be a locally conformally flat Riemannian manifold of even dimension \(n\geq 6\). Then \((M,g)\) is \(4\)-Einstein if and only if it is \(4\)-Thorpe.
\end{proposition}

\subsection{Hyper \((2k)\)-Einstein Metrics}

A Riemannian manifold \((M,g)\) is called \emph{hyper \((2k)\)-Einstein} (for \(2k<n\)) if the first contraction satisfies \(\mathbf{c}R^k=\lambda g^{2k-1}\) for some function \(\lambda\). This condition generalizes the Einstein condition (the case \(k=1\)). Hyper \((2k)\)-Einstein metrics are automatically both \((2k)\)-Einstein and \((2k)\)-Thorpe. For such metrics we prove a Schur-type result.

\begin{proposition}\label{prop-a}
Let \((M,g)\) be a hyper \((2k)\)-Einstein manifold of dimension \(n>2k\). If \eqref{Weak-Einstein} holds, then \(\|R^k\|\) is constant on \(M\).
\end{proposition}

A key technical tool in our analysis is a generalization of the classical Lanczos identity (which relates curvature invariants in dimension four) to higher dimensions. The following algebraic identity holds for symmetric double forms.

\begin{theorem}\label{thma}
Let \(\omega\) be a symmetric \((p,p)\) double form on an \(n\)-dimensional Euclidean space \((V,g)\) with \(n\geq 2p\geq 4\). Then
\begin{equation}\label{gen-ident}
\begin{split}
\frac12\frac{\mathbf{c}^{2p-1}(\omega^2)}{(2p-1)!}
&=(-1)^p\frac{\mathbf{c}^{p-1}(\omega\circ\omega)}{(p-1)!}
   +\frac{\mathbf{c}^{p}\omega}{p!}\frac{\mathbf{c}^{p-1}\omega}{(p-1)!} \\
&\quad+\sum_{r=1}^{p-1}(-1)^{r+p}\Bigl[
      \frac1{(r-1)!}\mathbf{c}^{r-1}\!\bigl(\iota_{\frac{\mathbf{c}^{r}\omega}{r!}}\omega\bigr)
     +\frac1{(p-r-1)!}\mathbf{c}^{p-r-1}\!\Bigl(\frac{\mathbf{c}^{r}\omega}{r!}
       \circ\frac{\mathbf{c}^{r}\omega}{r!}\Bigr)
   \Bigr].
\end{split}
\end{equation}
In the critical dimension \(n=2p\) the identity simplifies further (see Theorem~\ref{thma} for the precise statement).
\end{theorem}

For hyper \((2k)\)-Einstein metrics Theorem~\ref{thma} yields a concise relation.

\begin{proposition}\label{prop-b}
Let \(n\geq 4k\) and assume \(R^k\) satisfies \(\mathbf{c}R^k=\lambda g^{2k-1}\). Then
\[
\frac{\mathbf{c}^{4k-1}R^{2k}}{(4k-1)!}
   =\frac{\mathbf{c}^{2k-1}(R^k\circ R^k)}{(2k-1)!}+\lambda^2\,c(n,k)\,g,
\]
where \(c(n,k)\) is a constant depending only on \(n\) and \(k\).
\end{proposition}

As a consequence we obtain a striking equivalence for hyper \((2k)\)-Einstein metrics.

\begin{proposition}\label{prop-c}
Let \(n>4k\) and let \(g\) be a hyper \((2k)\)-Einstein metric. Then the following are equivalent:
\begin{itemize}
    \item \(g\) is a critical point of \(G_{2k}\);
    \item \(g\) is a \((4k)\)-Einstein metric;
    \item \(g\) is a weakly \((2k)\)-Einstein metric.
\end{itemize}
\end{proposition}
\subsection{Variational Tools: Two Key Lemmas}

A crucial technical aspect of our work is the development of efficient tools for computing first variations of curvature functionals within the double forms formalism. We present and apply two lemmas that streamline these calculations significantly.

The first lemma (Lemma~\ref{lemma:diff-metric}) gives a simple formula for the variation of the induced inner product on double forms. For a one-parameter family of metrics $g_t$ with $g_0=g$ and $g_0'=h$, and for any $(p,p)$ double forms $\omega_1,\omega_2$, we have
\[
\tilde{g}_0'(\omega_1,\omega_2)=-\tilde{g}_0\!\bigl(F_h(\omega_1),\omega_2\bigr),
\]
where $F_h$ is the derivation operator $F_h(\omega)=\frac{g^{p-1}h}{(p-1)!}\circ\omega+\omega\circ\frac{g^{p-1}h}{(p-1)!}$.

The second lemma (Lemma~\ref{lemma:diff-curvature}) provides an elegant expression for the linearisation of the Riemann curvature tensor $R$ (viewed as a $(2,2)$ double form) in the direction of $h$:
\[
R'h=-\frac14\bigl(D\tilde D+\tilde D D\bigr)(h)+\frac14 F_h(R),
\]
where $D$ and $\tilde D$ are the second Bianchi sum and its adjoint. This formula was first established in our earlier work \cite{Labbi-variational} and has proven instrumental in variational problems involving curvature functionals.

While Lemma~\ref{lemma:diff-curvature} originates from \cite{Labbi-variational}, Lemma~\ref{lemma:diff-metric} appears here in this general form for the first time. Together, these lemmas allow us to compute gradients of complicated curvature functionals with remarkable efficiency and to obtain compact, geometrically transparent formulas for the corresponding Euler--Lagrange equations. Their systematic application to higher-order Thorpe tensors and generalized Einstein conditions constitutes a significant methodological contribution of this paper.

\subsection{Plan of the Paper}

The paper is organized as follows. Section~2 reviews the necessary background on double forms, including inner products, contractions, the Hodge star operator, and the orthogonal decomposition of double forms. We also prove two key lemmas (Lemmas~\ref{lemma:diff-metric} and \ref{lemma:diff-curvature}) that provide powerful tools for computing first variations of curvature invariants; these lemmas streamline all subsequent variational calculations.

Section~3 is devoted to the proof of Theorem~\ref{thma}, our generalized Lanczos identity, and its corollaries. This identity plays a central role in later sections.

In Section~4 we study \((2k)\)-Einstein metrics, proving Proposition~\ref{prop-d} and other characterizations.

Section~5 introduces \((2k)\)-Thorpe and \((2k)\)-anti-Thorpe manifolds. We prove Propositions~\ref{prop-e}, \ref{prop-f}, \ref{prop-h}, \ref{prop-i}, and~\ref{prop-z}, as well as Corollaries~\ref{corollary-f} and~\ref{corollary-g}.

Finally, Section~6 computes the gradient of the functional \(G_{2k}\) and establishes Propositions~\ref{prop-j}, \ref{prop-b}, and~\ref{prop-c}, completing the proofs of the main results announced in this introduction.

\section{Double Forms: Background Material}\label{Background}

This section recalls the basic definitions and properties of double forms, a formalism that provides a powerful language for handling higher-order curvature tensors and their algebraic manipulations. We follow the notation and conventions established in \cite{Labbi-double-forms,Labbi-Lapl-Lich}.  

Let $(V, g)$ be a real Euclidean vector space of finite dimension $n$. Denote by
\[
\Lambda V^{*} = \bigoplus_{p \geq 0} \Lambda^{p}V^{*}
\]
the exterior algebra of the dual space $V^*$. The \emph{space of exterior double forms} on $V$ is defined as
\[
\mathcal{D}V^*:= \Lambda V^{*}\otimes \Lambda V^{*}
               =\bigoplus_{p,q\geq 0} \mathcal{D}^{p,q}V^*,
\qquad 
\mathcal{D}^{p,q}V^*:= \Lambda^{p}V^{*} \otimes \Lambda^{q}V^{*}.
\]

The space $\mathcal{D}V^*$ carries a natural bi-graded associative algebra structure, called the \emph{double exterior algebra of $V^*$}. The multiplication, called the \emph{exterior product}, is denoted by a dot (often omitted) and is induced by the wedge product on each factor.

\subsection{Inner product of double forms}

The metric $g$ extends naturally to inner products on $\Lambda V^{*}$ and $\mathcal{D} V^{*}=\Lambda V^{*}\otimes \Lambda V^{*}$, both denoted by $\langle\,\cdot\,,\cdot\,\rangle$. Explicitly, if $(e_i)$ is an orthonormal basis of $V$ with dual basis $(e^i)$, then
\[
\{e^{i_1}\wedge\dots\wedge e^{i_k}:i_1<\dots<i_k\}
\]
is declared an orthonormal basis of $\Lambda^k V^{*}$, and
\[
\{e^{i_1}\wedge\dots\wedge e^{i_p}\otimes e^{j_1}\wedge\dots\wedge e^{j_q}:
i_1<\dots<i_p,\;j_1<\dots<j_q\}
\]
is an orthonormal basis of $\mathcal{D}^{p,q} V^{*}$. Different degrees are orthogonal: $\Lambda^k V^{*}\perp \Lambda^l V^{*}$ for $k\neq l$, and $\mathcal{D}^{p,q} V^{*}\perp \mathcal{D}^{r,s} V^{*}$ whenever $p\neq r$ or $q\neq s$.

If $g^p$ denotes the $p$-th exterior power of $g$ in $\mathcal{D} V^{*}$, then
\[
\Bigl\langle \frac{g^p}{p!}, \frac{g^p}{p!}\Bigr\rangle=\binom{n}{p},
\qquad 
\langle g^p, g^r\rangle=0\;\; \text{if}\;\; p\neq r.
\]

\subsection{Interior product, contraction and the Hodge star operator}\label{Hodge}

For $\omega \in \mathcal{D}V^*$, let $\mu_{\omega}$ be the left exterior multiplication by $\omega$:
\[
\mu_{\omega}(\alpha)=\omega.\alpha,\qquad \alpha\in\mathcal{D} V^{*}.
\]

The \emph{interior product} $\iota_{\omega}$ is the adjoint of $\mu_{\omega}$:
\[
\langle \iota_{\omega}\alpha,\beta\rangle:=\langle \alpha,\omega.\beta\rangle,
\qquad \alpha,\beta\in\mathcal{D} V^{*}.
\]

The \emph{contraction map} $\cc$ is the adjoint of left multiplication by the metric $g$, viewed as a $(1,1)$ double form:
\[
\cc\omega:=\iota_g\omega.
\]
We write $\cc^k$ for the $k$-fold composition of $\cc$.

The (double) \emph{Hodge star operator} on double forms is the isomorphism
\[
\str: \mathcal{D} V^{*}\to \mathcal{D} V^{*},\qquad 
\str\omega =\iota_{\omega}\frac{g^n}{n!}.
\]

For a simple double form $\omega=\theta_1\otimes \theta_2$ one has \cite{Labbi-Lapl-Lich}
\[
\str(\theta_1\otimes \theta_2)=\str\theta_1\otimes \str\theta_2,
\]
where the right-hand side uses the ordinary Hodge star on forms. This relation was taken as the definition of $\str$ in \cite{Labbi-double-forms}.\\
We mention for future reference a simple but important cancellation property of the left multiplication map by an exterior power of the metric.

\begin{proposition}[{\cite[Proposition 2.3]{Labbi-double-forms}}]\label{cor:cancellation}
For a $(p,q)$ double form $\omega$ with $p+q+k \leq n$, one has
\[
g^k\omega = 0 \;\Longrightarrow\; \omega = 0.
\]
\end{proposition}

\subsection{Composition product of double forms}

Via the Euclidean structure, $\mathcal{D} V^{*}$ is canonically isomorphic to the space of linear maps $L(\Lambda V, \Lambda V)$. Pulling back the composition of maps endows $\mathcal{D} V^{*}$ with a second associative product, called the \emph{composition product} and denoted by $\circ$ (see \cite{Labbi-Archivum,Labbi-JAUMS,Labbi-Bel}).

\subsection{Orthogonal decomposition of double forms}\label{ortho-decomp}

A central algebraic tool in our work is the orthogonal decomposition of a double form into trace‑free components. A $(p,p)$ double form $\omega$ is called \emph{trace‑free} if $\cc\omega = 0$.

\begin{proposition}[\cite{Kulkarni,Labbi-double-forms}]\label{prop:orth-decomp}
Let $\omega$ be a $(p,p)$ double form on $(V,g)$ and set $k = \min\{p, n-p\}$. 
There exist unique trace‑free $(i,i)$ double forms $\omega_i$ ($i=0,1,\dots,k$) such that
\begin{equation}\label{orth-decomp}
\omega = \sum_{i=0}^k g^{p-i}\omega_i.
\end{equation}
In particular, $\omega = 0$ if and only if $\omega_i = 0$ for all $i$ with $0 \leq i \leq \min\{p, n-p\}$.
\end{proposition}

When $n\geq 2p$, the components $\omega_i$ are uniquely determined by the successive contractions of $\omega$; explicit formulas are given in \cite[Theorem 3.7]{Labbi-double-forms}. 
For $n<2p$, the representation \eqref{orth-decomp} simplifies: by \cite[Proposition 2.1]{Labbi-Balkan} one has $\omega_k=0$ whenever $n-p<k\leq p$, and there exists a unique $(n-p,n-p)$ double form $\bar{\omega}$ such that $\omega = g^{2p-n}\bar{\omega}$. The remaining components of $\omega$ are then obtained from those of $\bar{\omega}$ via the decomposition for the case $n\geq 2(n-p)$.

Recall that for the Riemann curvature tensor $R$ (a $(2,2)$ double form) in dimension $n \geq 4$, the orthogonal decomposition \eqref{orth-decomp} takes the form
\[
R = \omega_2 + g\omega_1 + g^2\omega_0,
\]
where $\omega_2$ is the Weyl tensor, $\omega_1$ is the traceless Ricci part, and $\omega_0$ is the scalar curvature part. 
In this decomposition, the condition $R = gA$ for some $(1,1)$ form $A$ is equivalent to the vanishing the Weyl tensor tensor, i.e., $\omega_2 \equiv 0$. 
Furthermore, a metric $g$  is Einstein ($\cc R = \lambda g$) if and only if the traceless Ricci part vanishes, i.e., $\omega_1 \equiv 0$.

The following corollary generalizes these facts to higher-order double forms.

\begin{corollary}\label{cor:vanishing-omega_i}
Let $n \geq 2p$ and let $\omega$ be a $(p,p)$ double form with orthogonal decomposition as in \eqref{orth-decomp}. For $0 \leq k \leq p-1$, one has
\[
\cc^k\omega = gA \;\; \text{for some double form } A
\quad\Longleftrightarrow\quad \omega_{p-k} = 0.\]
\end{corollary}
We remark that for $k=p$, we clearly have $\cc^p\omega=0 \iff \omega_0=0.$
\begin{proof}
Formula (12) in \cite{Labbi-double-forms} gives the orthogonal decomposition of $\cc^k\omega$:
\[
\cc^k\omega = \alpha_{n,k,p}\,\omega_{p-k} + gA,
\]
where $\alpha_{n,k,p} \neq 0$ and $A$ is a $(p-k-1,p-k-1)$ double form. The result follows from the uniqueness of the decomposition.
\end{proof}
\subsection{Two key lemmas for variational calculations}

The following operator, introduced in \cite{Labbi-variational}, greatly simplifies the computation of first variations of curvature invariants. For a symmetric $(1,1)$ double form $h$, define
\[
F_h(k)=h\circ k+k\circ h,\qquad k\in\mathcal{D}^{1,1}V^*,
\]
and extend $F_h$ as a derivation to the whole diagonal subalgebra $\bigoplus_{p\ge0}\mathcal{D}^{p,p}V^*$. For a $(p,p)$ double form $\omega$ one has \cite{Labbi-Lapl-Lich}
\begin{equation}\label{Fh-formula}
F_h(\omega)=\frac{g^{p-1}h}{(p-1)!}\circ\omega
            +\omega\circ\frac{g^{p-1}h}{(p-1)!}.
\end{equation}

The operator $F_h$ is self‑adjoint and acts as a derivation on the exterior algebra of double forms. Its relevance for variational problems is captured by the next lemma.

\begin{lemma}\label{lemma:diff-metric}
Let $g_t$ be a smooth family of inner products on $V$ with $g_0=g$ and $g'_0=h$. Denote by $\tilde{g}_t$ the induced inner product on $(p,p)$ double forms. Then for any $(p,p)$ double forms $\omega_1,\omega_2$,
\[
\tilde{g}'_0(\omega_1,\omega_2)=
-\tilde{g}_0\!\bigl(F_h(\omega_1),\omega_2\bigr).\]
\end{lemma}
\begin{proof}
Without loss of generality, suppose $\omega_1 = \theta_1 \otimes \theta_2$ and $\omega_2 = \theta_3 \otimes \theta_4$. Recall that the exterior power $g^p/p!$ of an inner product $g$ coincides with the standard induced product on $p$-vectors. For a $p$-form $\theta$, let $\theta^{\sharp_t}$ denote the corresponding $p$-vector via the $\sharp$-isomorphism induced by $g_t^p/p!$. Then,
\begin{equation}\label{eq1}
\tilde{g}_t(\omega_1, \omega_2) = \frac{g_t^p}{p!}(\theta_1^{\sharp_t}, \theta_3^{\sharp_t}) \cdot \frac{g_t^p}{p!}(\theta_2^{\sharp_t}, \theta_4^{\sharp_t}).
\end{equation}

Differentiating the identity $\frac{g_t^p}{p!}(\theta_1^{\sharp_t}, \cdot) = \theta_1(\cdot)$ yields
\[
\frac{g^p}{p!}\left(\frac{d}{dt}\Big|_{t=0} \theta_1^{\sharp_t}, \cdot\right) = -\frac{g^{p-1} h}{(p-1)!}(\theta_1^\sharp, \cdot),
\]
where $\sharp = \sharp_0$. Differentiating both sides of Equation \eqref{eq1} gives:
\[
\tilde{g}_0'(\omega_1, \omega_2) = -\frac{g^{p-1} h}{(p-1)!}(\theta_1^\sharp, \theta_3^\sharp) \cdot \frac{g^p}{p!}(\theta_2^\sharp, \theta_4^\sharp) - \frac{g^p}{p!}(\theta_1^\sharp, \theta_3^\sharp) \cdot \frac{g^{p-1} h}{(p-1)!}(\theta_2^\sharp, \theta_4^\sharp).
\]

On the other hand, a direct computation shows:
\begin{align*}
\tilde{g}_0\bigl(F_h(\omega_1), \omega_2\bigr) 
&= \tilde{g}_0\left( \frac{g^{p-1} h}{(p-1)!} \circ (\theta_1 \otimes \theta_2) + (\theta_1 \otimes \theta_2) \circ \frac{g^{p-1} h}{(p-1)!}, \theta_3 \otimes \theta_4 \right) \\
&= \frac{g^{p-1} h}{(p-1)!}(\theta_1^\sharp, \theta_3^\sharp) \cdot \frac{g^p}{p!}(\theta_2^\sharp, \theta_4^\sharp) + \frac{g^p}{p!}(\theta_1^\sharp, \theta_3^\sharp) \cdot \frac{g^{p-1} h}{(p-1)!}(\theta_2^\sharp, \theta_4^\sharp).
\end{align*}
This completes the proof.
\end{proof}

The second lemma, proved in \cite{Labbi-variational}, gives a clean formula for the linearisation of the Riemann curvature tensor as the sum of a differential term and an algebraic term.

\begin{lemma}[\cite{Labbi-variational}]\label{lemma:diff-curvature}
Let $R$ be the Riemann curvature tensor viewed as a symmetric $(2,2)$ double form. Its directional derivative in the direction of a symmetric $(1,1)$ form $h$ is
\[
R'h=-\frac14\bigl(D\tilde D+\tilde D D\bigr)(h)
      +\frac14 F_h(R),
\]
where $D$ and $\tilde D$ denote the second Bianchi sum and its adjoint, respectively.
\end{lemma}

These two lemmas will be used repeatedly in the computation of gradients of the functionals studied in the paper.
\section{Generalized Lanczos identities}

This section is devoted to a generalization of the classical Lanczos identity. We first recall the elementary but important fact that on an even-dimensional Riemannian manifold, the top exterior power of the curvature tensor is essentially a multiple of the volume double form. From this observation we derive a higher‑dimensional analogue of the familiar two‑dimensional identity $\cc R = \frac{\cc^2R}{2}\,g$. We then prove an identity that relates the tensor $\cc^{2k-1}(R^k\circ R^k)$  and $\cc^{4k-1}(R^{2k})$. This identity, which reduces to the classical Lanczos identity in dimension four, is the main result of the section and will play a crucial role in the characterization of weakly $(2k)$-Einstein metrics.

\subsection{The top exterior power in even dimensions}

Let $(M,g)$ be a Riemannian manifold of even dimension $n = 2p$. The $(2p)$-Thorpe tensor $R^p$ is an $(n,n)$‑double form, as is $g^n$. By dimensionality they must be proportional; hence there exists a scalar function $\lambda$ on $M$ such that
\[
R^p = \lambda\,\frac{g^n}{n!}.
\]
Taking the inner product with $g^n/n!$ and using that $\cc = \iota_g$ is the adjoint of the left exterior multiplication by $g$, we obtain
\[
\lambda = \bigl\langle R^p, \tfrac{g^n}{n!}\bigr\rangle = \frac{\cc^{\,n}R^p}{n!}.
\]
Thus
\begin{equation}\label{eq:top-power}
R^p = \frac{\cc^{2p}R^p}{(2p)!}\,\frac{g^n}{n!}.
\end{equation}
Contracting \eqref{eq:top-power} $(2p-1)$ times yields
\begin{equation}\label{eq:Tn=0}
\cc^{2p-1}R^p = \frac{\cc^{2p}R^p}{2p}\;g .
\end{equation}
For $p=1$ this is the familiar two‑dimensional relation $\cc R = \frac{\cc^2R}{2}\,g$, which expresses the vanishing of the Einstein tensor in dimension two. In general, \eqref{eq:Tn=0} is equivalent to the vanishing of the Lovelock tensor
\[
T_{2p}:=\frac{\cc^{2p-1}R^p}{(2p-1)!}-\frac{\cc^{2p}R^p}{(2p)!}\,g
\]
in dimension $n=2p$. These apparently simple identities are in fact linearized versions of the Gauss–Bonnet theorem (see \cite{Labbi-variational,Labbi-Archivum}).

\subsection{Lanczos and Avez identities in dimension four}

In all dimensions $n\ge 4$, we show that the curvature tensor satisfies the identity
\begin{equation}\label{eq:R-identity}
\frac{\cc^3R^2}{3!}=2\cc(R\circ R)+\Scal\cdot\Ric-2\iota_{\Ric}R-2(\Ric\circ\Ric),
\end{equation}
which is equivalent to formula (7.3) in \cite{Patterson}. Contracting once gives Avez' formula \cite{Avez}:
\begin{equation}\label{eq:Avez}
\frac{\cc^4R^2}{4!}= \|R\|^2+\frac14\Scal^2-\|\Ric\|^2.
\end{equation}
Indeed,
\begin{align*}
\cc^2(R\circ R) &= \langle R\circ R,g^2\rangle
                 = \langle R,R\circ g^2\rangle
                 = 2\langle R,R\rangle,\\[2mm]
\cc\bigl(\iota_{\Ric}R\bigr) &= \langle \iota_{\Ric}R,g\rangle
                             = \langle R,\Ric.g\rangle
                             = \langle R,g.\Ric\rangle
                             = \langle \Ric,\Ric\rangle,\\[2mm]
\cc(\Ric\circ\Ric) &= \langle \Ric\circ\Ric,g\rangle
                   = \langle \Ric,\Ric\circ g\rangle
                   = \langle \Ric,\Ric\rangle .
\end{align*}

When $n=4$, equation \eqref{eq:Tn=0} with $p=2$ gives
\[
\frac{\cc^3R^2}{3!}= \Bigl(\frac{\cc^4R^2}{4!}\Bigr)g .
\]
Substituting \eqref{eq:Avez} and \eqref{eq:R-identity} we obtain Lanczos' identity \cite{Lanczos}:
\begin{equation}\label{eq:Lanczos}
\Bigl(\|R\|^2-\|\Ric\|^2+\frac{\Scal^2}{4}\Bigr)g
   = 2\cc(R\circ R)+\Scal\cdot\Ric-2\iota_{\Ric}R-2(\Ric\circ\Ric).
\end{equation}
Lanczos used this identity to show that the quadratic Lagrangian given by the Gauss–Bonnet curvature does not contribute to the equations of motion in four dimensions—an early manifestation of the Gauss–Bonnet theorem in that dimension.

\subsection{A general algebraic identity}

The following theorem extends \eqref{eq:Lanczos} to symmetric double forms of arbitrary even degree. It is the central algebraic result of this section.

\begin{theorem}\label{thm:gen-Lanczos}
Let $\omega$ be a symmetric $(p,p)$ double form on an $n$-dimensional Euclidean space $(V,g)$ with $n\ge 2p\ge 4$. Then
\begin{equation}\label{eq:gen-ident}
\begin{split}
\frac12\frac{\cc^{2p-1}(\omega^2)}{(2p-1)!}
&=(-1)^p\frac{\cc^{p-1}(\omega\circ\omega)}{(p-1)!}
   +\frac{\cc^p\omega}{p!}\,\frac{\cc^{p-1}\omega}{(p-1)!}\\
&\quad+\sum_{r=1}^{p-1}(-1)^{r+p}\Bigl[
      \frac1{(r-1)!}\cc^{r-1}\!\bigl(\iota_{\frac{\cc^{r}\omega}{r!}}\omega\bigr)
     +\frac1{(p-r-1)!}\cc^{p-r-1}\!\Bigl(
        \frac{\cc^{r}\omega}{r!}\circ\frac{\cc^{r}\omega}{r!}\Bigr)
    \Bigr].
\end{split}
\end{equation}
If, in addition, $n=2p$, then
\begin{equation}\label{eq:gen-ident-2p}
\begin{split}
\Bigl(\sum_{r=0}^{p}\frac{(-1)^{r+p}}{2}
      \Bigl\|\frac{\cc^{r}\omega}{r!}\Bigr\|^2\Bigr) g
&=(-1)^p\frac{\cc^{p-1}(\omega\circ\omega)}{(p-1)!}
   +\frac{\cc^p\omega}{p!}\,\frac{\cc^{p-1}\omega}{(p-1)!}\\
&\quad+\sum_{r=1}^{p-1}(-1)^{r+p}\Bigl[
      \frac1{(r-1)!}\cc^{r-1}\!\bigl(\iota_{\frac{\cc^{r}\omega}{r!}}\omega\bigr)
     +\frac1{(p-r-1)!}\cc^{p-r-1}\!\Bigl(
        \frac{\cc^{r}\omega}{r!}\circ\frac{\cc^{r}\omega}{r!}\Bigr)
    \Bigr].
\end{split}
\end{equation}
\end{theorem}

\begin{proof}
Let $h$ be an arbitrary symmetric $(1,1)$ bilinear form. Using successively that $\cc$ is the adjoint of left exterior multiplication by $g$, that $F_h$ is self‑adjoint and acts as a derivation on the exterior algebra, we compute
\begin{align*}
\Bigl\langle\frac{\cc^{2p-1}(\omega^2)}{(2p-1)!},h\Bigr\rangle
&=\bigl\langle\omega^2,\tfrac{g^{2p-1}h}{(2p-1)!}\bigr\rangle
  =\bigl\langle\omega^2,\tfrac12F_h\bigl(\tfrac{g^{2p}}{(2p)!}\bigr)\bigr\rangle
  =\bigl\langle\omega\,F_h(\omega),\tfrac{g^{2p}}{(2p)!}\bigr\rangle \\
&=\bigl\langle F_h(\omega),\iota_\omega\bigl(\tfrac{g^{2p}}{(2p)!}\bigr)\bigr\rangle
  =\bigl\langle F_h(\omega),\ast\bigl(\tfrac{g^{n-2p}}{(n-2p)!}\,\omega\bigr)\bigr\rangle .
\end{align*}\\
In the last two steps we used the fact that the interior product $\iota_{\omega}$ is the adjoint of the exterior multiplication by $\omega$ and  formula (16) in \cite{Labbi-Bel}.\\
Formula (15) of \cite{Labbi-double-forms} (or formula (25) of \cite{Labbi-Bel}, which does not require the first Bianchi identity) gives
\[
\ast\bigl(\tfrac{g^{n-2p}}{(n-2p)!}\,\omega\bigr)
   =\sum_{r=0}^{p}(-1)^{r+p}\frac{g^r}{r!}\,\frac{\cc^r\omega}{r!}.
\]
Hence
\begin{align*}
\Bigl\langle\frac{\cc^{2p-1}(\omega^2)}{(2p-1)!},h\Bigr\rangle
&=\sum_{r=0}^{p}(-1)^{r+p}
   \bigl\langle\omega,F_h\bigl(\tfrac{g^r}{r!}\,\tfrac{\cc^r\omega}{r!}\bigr)\bigr\rangle \\
&=(-1)^p\langle\omega,F_h(\omega)\rangle
   +\bigl\langle\omega,F_h\bigl(\tfrac{g^p}{p!}\,\tfrac{\cc^p\omega}{p!}\bigr)\bigr\rangle\\
&\quad+\sum_{r=1}^{p-1}(-1)^{r+p}
   \bigl\langle\omega,
         2\tfrac{g^{r-1}h}{(r-1)!}.\tfrac{\cc^r\omega}{r!}
        +\tfrac{g^r}{r!}F_h\bigl(\tfrac{\cc^r\omega}{r!}\bigr)\bigr\rangle .
\end{align*}
For the last sum we use
\[
F_h\bigl(\tfrac{\cc^r\omega}{r!}\bigr)
   =\tfrac{g^{p-r-1}h}{(p-r-1)!}\circ\tfrac{\cc^r\omega}{r!}
    +\tfrac{\cc^r\omega}{r!}\circ\tfrac{g^{p-r-1}h}{(p-r+1)!},
\]
and obtain
\begin{align*}
\Bigl\langle\frac{\cc^{2p-1}(\omega^2)}{(2p-1)!},h\Bigr\rangle
&=(-1)^p\langle\omega,F_h(\omega)\rangle
   +\tfrac{\cc^p\omega}{p!}\bigl\langle\omega,F_h\bigl(\tfrac{g^p}{p!}\bigr)\bigr\rangle\\
&\quad+\sum_{r=1}^{p-1}(-1)^{r+p}
   \Bigl\langle2\tfrac{\cc^{r-1}}{(r-1)!}
          \bigl(\iota_{\frac{\cc^r\omega}{r!}}\omega\bigr)
         +2\tfrac{\cc^{p-r-1}}{(p-r-1)!}
          \bigl(\tfrac{\cc^r\omega}{r!}\circ\tfrac{\cc^r\omega}{r!}\bigr),\,h\Bigr\rangle .
\end{align*}
Finally,
\[
\langle\omega,F_h(\omega)\rangle
   =\bigl\langle\omega,
      \tfrac{g^{p-1}h}{(p-1)!}\circ\omega
      +\omega\circ\tfrac{g^{p-1}h}{(p-1)!}\bigr\rangle
   =\bigl\langle2\tfrac{\cc^{p-1}(\omega\circ\omega)}{(p-1)!},\,h\bigr\rangle .
\]
Since $h$ was arbitrary, equation \eqref{eq:gen-ident} follows.

When $n=2p$, the form $\omega^2$ is an $(n,n)$-double form and therefore proportional to $g^n$:
\[
\omega^2=\frac{\cc^n(\omega^2)}{n!}\,\frac{g^n}{n!},
\qquad\text{so}\qquad
\frac{\cc^{n-1}(\omega^2)}{(n-1)!}= \frac{\cc^n(\omega^2)}{n!}\,g .
\]
Moreover,
\begin{align*}
\frac{\cc^n(\omega^2)}{n!}
&=\bigl\langle\omega^2,\tfrac{g^n}{n!}\bigr\rangle
  =\bigl\langle\omega,\iota_\omega\bigl(\tfrac{g^n}{n!}\bigr)\bigr\rangle
  =\langle\omega,\ast\omega\rangle \\
&=\Bigl\langle\omega,\sum_{r=0}^{p}(-1)^{r+p}
      \tfrac{g^r}{r!}\,\tfrac{\cc^r\omega}{r!}\Bigr\rangle
  =\sum_{r=0}^{p}(-1)^{r+p}
    \bigl\langle\tfrac{\cc^r\omega}{r!},\tfrac{\cc^r\omega}{r!}\bigr\rangle .
\end{align*}
Substituting this expression for $\cc^n(\omega^2)/n!$ into the previous formula yields \eqref{eq:gen-ident-2p}.
\end{proof}

\begin{corollary}\label{cor:tracefree}
Let $W$ be a trace‑free symmetric $(p,p)$ double form on $(V,g)$ with $n\ge 2p\ge 4$. Then
\[
\frac12\frac{\cc^{2p-1}(W^2)}{(2p-1)!}=(-1)^p\frac{\cc^{p-1}(W\circ W)}{(p-1)!}.
\]
If in addition $n=2p$, then
\[
\frac12\|W\|^2=\frac{\cc^{p-1}(W\circ W)}{(p-1)!}.
\]
\end{corollary}

\subsection{Reformulation via the Hodge star}

For a symmetric $(p,p)$ double form with $2p+1\le n$ one has the identity (see \cite{Labbi-double-forms,Labbi-Bel})
\begin{equation}\label{eq:Hodge-square}
\ast\Bigl(\frac{g^{n-2p-1}\omega^2}{(n-2p-1)!}\Bigr)
   =\frac{\cc^{2p}\omega^2}{(2p)!}\,g
    -\frac{\cc^{2p-1}\omega^2}{(2p-1)!}.
\end{equation}
Using formula (60) of \cite{Labbi-Archivum},
\[
\frac{\cc^{2p}\omega^2}{(2p)!}
   =\sum_{r=0}^{p}(-1)^{r+p}\Bigl\|\frac{\cc^r\omega}{r!}\Bigr\|^2
   =(-1)^p\|\omega\|^2+\sum_{r=1}^{p}(-1)^{r+p}\Bigl\|\frac{\cc^r\omega}{r!}\Bigr\|^2,
\]
together with Theorem~\ref{thm:gen-Lanczos}, equation \eqref{eq:Hodge-square} can be rewritten as
\begin{equation}\label{eq:gen-ident-star}
\begin{split}
\ast\Bigl(\frac{g^{n-2p-1}\omega^2}{(n-2p-1)!}\Bigr)
&=\sum_{r=0}^{p}(-1)^{r+p}\Bigl\|\frac{\cc^r\omega}{r!}\Bigr\|^2\,g
   -2(-1)^p\frac{\cc^{p-1}(\omega\circ\omega)}{(p-1)!}
   -2\frac{\cc^p\omega}{p!}\,\frac{\cc^{p-1}\omega}{(p-1)!}\\
&\quad-2\sum_{r=1}^{p-1}(-1)^{r+p}\Bigl[
      \frac1{(r-1)!}\cc^{r-1}\!\bigl(\iota_{\frac{\cc^{r}\omega}{r!}}\omega\bigr)
     +\frac1{(p-r-1)!}\cc^{p-r-1}\!\Bigl(
        \frac{\cc^{r}\omega}{r!}\circ\frac{\cc^{r}\omega}{r!}\Bigr)
    \Bigr].
\end{split}
\end{equation}

Now assume that $\omega$ satisfies $\cc\omega=\lambda\frac{g^{p-1}}{(p-1)!}$ for some constant $\lambda$. Then for every $r$ with $1\le r\le p$,
\[
\cc^r\omega=\frac{(n-p+r)!\,\lambda}{(p-r)!\,(n-p+1)!}\,g^{p-r}.
\]
Inserting these expressions into \eqref{eq:gen-ident-star} we obtain a much simpler relation.

\begin{corollary}\label{cor:hyper-identity}
Let $\omega$ be a symmetric $(p,p)$ double form with $5\le 2p+1\le n$ and suppose
$\cc\omega=\lambda\frac{g^{p-1}}{(p-1)!}$ for some constant $\lambda$. Then
\begin{equation}\label{eq:hyper-identity}
(-1)^p\ast\Bigl(\frac{g^{n-2p-1}\omega^2}{(n-2p-1)!}\Bigr)
   =\|\omega\|^2g-2\frac{\cc^{p-1}(\omega\circ\omega)}{(p-1)!}
    +\lambda^2\,c(n,p)\,g,
\end{equation}
where $c(n,p)$ is a numerical constant depending only on $n$ and $p$.
\end{corollary}

\subsection{Schur-type results for double forms satisfying the second Bianchi's identity}

When a double form satisfies the second Bianchi identity, the algebraic identities above imply rigidity properties reminiscent of the classical Schur lemma.

\begin{theorem}\label{thm:Schur-properties}
Let $\omega$ be a symmetric $(p,p)$ double form satisfying the second Bianchi identity on a compact and connected Riemannian $n$-manifold $(M,g)$.
\begin{enumerate}
\item If $\cc^{\,p-1}\omega=\lambda g$ for some function $\lambda$ and $2\le p<n$, then $\lambda$ is constant on~$M$.
\item If $\cc\omega=\lambda\frac{g^{p-1}}{(p-1)!}$ and $\cc^{p-1}(\omega\circ\omega)=\mu g$ for some functions $\lambda,\mu$ and $4\le 2p<n$, then both $\lambda$ and $\mu$ are constant on~$M$.
\end{enumerate}
\end{theorem}

\begin{proof}
(1)  From \cite[Formula (15)]{Labbi-double-forms} (see also \cite{Labbi-Bel}) we have
\[
\ast\Bigl(\frac{g^{n-p-1}\omega}{(n-p-1)!}\Bigr)
   =\frac{\cc^{\,p}\omega}{p!}\,g-\frac{\cc^{\,p-1}\omega}{(p-1)!}.
\]
Under the hypothesis $\cc^{\,p-1}\omega=\lambda g$ one has $\cc^p\omega=\lambda n$ and  this becomes
\[
\ast\Bigl(\frac{g^{n-p-1}\omega}{(n-p-1)!}\Bigr)
   =\lambda\Bigl(\frac{n-p}{p!}\Bigr)g .
\]
Because $\omega$ satisfies the second Bianchi identity, the left‑hand side is divergence‑free; consequently the right‑hand side is also divergence‑free, which forces $\lambda$ to be constant.

(2)  The condition $\cc\omega=\lambda\frac{g^{p-1}}{(p-1)!}$ implies, via part (1), that $\lambda$ is constant. Moreover, equation \eqref{eq:hyper-identity} shows that the tensor
\[
(-1)^p\ast\Bigl(\frac{g^{n-2p-1}\omega^2}{(n-2p-1)!}\Bigr)
   =\|\omega\|^2g-2\frac{\cc^{p-1}(\omega\circ\omega)}{(p-1)!}
    +\lambda^2c(n,p)g
\]
is divergence‑free (again because $\omega$ satisfies the second Bianchi identity). Using $\cc^{p-1}(\omega\circ\omega)=\mu g$ and contracting once, we find $\mu=\frac{p}{n}\|\omega\|^2$. Hence
\[
\cc^{p-1}(\omega\circ\omega)=\frac{p}{n}\|\omega\|^2 g,
\qquad\text{and}\qquad
\delta\bigl(\cc^{p-1}(\omega\circ\omega)\bigr)=\frac{p}{n}\,\tilde{D}(\|\omega\|^2).
\]
On the other hand, from the divergence‑free property of the left‑hand side of \eqref{eq:hyper-identity} we obtain
\[
\delta\bigl(\cc^{p-1}(\omega\circ\omega)\bigr)=\frac12\tilde{D}(\|\omega\|^2).
\]
Comparing the two expressions gives $\frac{n-2p}{2n}\tilde{D}(\|\omega\|^2)=0$, so $\|\omega\|^2$ (and therefore $\mu$) is constant.
\end{proof}

\section{Critical metrics for the Lovelock functionals}

Let $(M,g)$ be a Riemannian manifold of dimension $n$ and let $0 \leq 2k \leq n$.
Denote by $R^k$ the $k$-th exterior power of the Riemann curvature tensor $R$, viewed as a $(2,2)$ double form.

\begin{definition}\label{GB-definition}
\begin{enumerate}
\item The \emph{$(2k)$-th Gauss--Bonnet curvature} of $(M,g)$ is the full contraction of $R^k$:
\begin{equation}\label{eq:h2k}
h_{2k} = \frac{\cc^{2k}R^k}{(2k)!}.
\end{equation}

\item The \emph{$(2k)$-th Lovelock tensor} of $(M,g)$ is
\begin{equation}\label{eq:T2k}
T_{2k} = h_{2k} g - \frac{\cc^{2k-1}R^k}{(2k-1)!}.
\end{equation}
\end{enumerate}
\end{definition}

For $k=0$ we have $h_0=1$ and $T_0=g$; for $k=1$ we recover one-half of the scalar curvature and the Einstein tensor.
When $2k=n$, we have $R^k = h_n \frac{g^n}{n!}$ and consequently $T_n = 0$. Up to a constant factor, $h_n$ is the integrand in the Gauss--Bonnet formula.

The following theorem, first proved by Lovelock \cite{Lovelock}, describes the gradient of the Lovelock functional.

\begin{theorem}[\cite{Lovelock, Labbi-variational}]\label{thm:grad-H2k}
For $2 \leq 2k \leq n$, the gradient of the functional
\[
H_{2k}(g) = \int_M h_{2k}(g) \, \mathrm{dvol}_g
\]
on the space of Riemannian metrics on $M$ equals one half of the $(2k)$-th Lovelock tensor $T_{2k}$.
\end{theorem}

For completeness we provide a proof using Lemma~\ref{lemma:diff-metric}. This proof follows the lines of \cite{Labbi-variational}.

\begin{proof}
Write
\[
H_{2k}(g) = \int_M \tilde{g}\!\left(R^k, \frac{g^{2k}}{(2k)!}\right) \mathrm{dvol}_g .
\]
Using Lemma~\ref{lemma:diff-metric}, the directional derivative of $H_{2k}$ in the direction of a symmetric $(1,1)$ double form $h$ is
\begin{align*}
H_{2k}'.h &= -\bigl\langle F_h(R^k), \tfrac{g^{2k}}{(2k)!}\bigr\rangle
           + \bigl\langle kR^{k-1}R'h, \tfrac{g^{2k}}{(2k)!}\bigr\rangle 
           + \bigl\langle R^k, \tfrac{g^{2k-1}h}{(2k-1)!}\bigr\rangle
           + \tfrac12\bigl\langle h_{2k}g, h\bigr\rangle \\
          &= I + II + III + IV,
\end{align*}
where $\langle\cdot,\cdot\rangle = \tilde{g}(\cdot,\cdot)$.

Because $F_h$ is self‑adjoint and acts as a derivation,
\[
I = -\bigl\langle R^k, F_h\bigl(\tfrac{g^{2k}}{(2k)!}\bigr)\bigr\rangle
   = -\bigl\langle R^k, \tfrac{g^{2k-1}}{(2k-1)!}F_h(g)\bigr\rangle
   = -2\bigl\langle R^k, \tfrac{g^{2k-1}}{(2k-1)!}h\bigr\rangle
   = -2\bigl\langle \tfrac{\cc^{2k-1}R^k}{(2k-1)!}, h\bigr\rangle .
\]

For the second term we use the expression for $R'h$ from Lemma~\ref{lemma:diff-curvature}:
\begin{align*}
II &= \tfrac{k}{4}\bigl\langle R^{k-1}F_h(R), \tfrac{g^{2k}}{(2k)!}\bigr\rangle
     -\tfrac{k}{4}\bigl\langle R^{k-1}(D\tilde D+\tilde D D)(h), \tfrac{g^{2k}}{(2k)!}\bigr\rangle \\
   &= \tfrac14\bigl\langle F_h(R^k), \tfrac{g^{2k}}{(2k)!}\bigr\rangle
      -\tfrac{k}{4}\bigl\langle (D\tilde D+\tilde D D)(h), \iota_{R^{k-1}}\tfrac{g^{2k}}{(2k)!}\bigr\rangle \\
   &= \tfrac12\bigl\langle R^k, F_h\bigl(\tfrac{g^{2k}}{(2k)!}\bigr)\bigr\rangle
      -\tfrac{k}{4}\bigl\langle h,
        (\delta\tilde\delta+\tilde\delta\delta)\bigl(
        \iota_{R^{k-1}}\tfrac{g^{2k}}{(2k)!}\bigr)\bigr\rangle .
\end{align*}
In the second line we used that $\iota_\omega$ is the adjoint of exterior multiplication by $\omega$; in the third line we used that $\delta\tilde\delta+\tilde\delta\delta$ is the adjoint of $D\tilde D+\tilde D D$ with respect to the integral scalar product.

Formula (16) of \cite{Labbi-Bel}  gives
\[
\iota_{R^{k-1}}\tfrac{g^{2k}}{(2k)!}
   = \ast\!\left(\tfrac{g^{n-2k}}{(n-2k)!}R^{k-1}\right).
\]
Because $R^{k-1}$ satisfies the second Bianchi identity, the right‑hand side is divergence‑free. Hence $II = \tfrac12\bigl\langle R^k, F_h\bigl(\tfrac{g^{2k}}{(2k)!}\bigr)\bigr\rangle = -\tfrac12 I$.

Collecting the terms,
\[
H_{2k}'.h = I - \tfrac12 I + III + IV
          = -\tfrac12\bigl\langle R^k,\tfrac{g^{2k-1}h}{(2k-1)!}\bigr\rangle
            + \tfrac12\bigl\langle h_{2k}g,h\bigr\rangle
          = \tfrac12\bigl\langle h_{2k}g-\tfrac{\cc^{2k-1}R^k}{(2k-1)!}, h\bigr\rangle.\]
Thus $\nabla H_{2k} = \tfrac12 T_{2k}$.
\end{proof}

\begin{definition}[\cite{Patterson, Labbi-variational}]\label{def:2k-Einstein}
A Riemannian metric is called \emph{$(2k)$-Einstein} if it is a critical point of $H_{2k}$ restricted to the space of unit‑volume metrics.
\end{definition}

Consequently, a metric is $(2k)$-Einstein precisely when its Lovelock tensor is proportional to the metric: $T_{2k} = \lambda g$ for some function $\lambda$. For $k=1$ this reduces to the usual Einstein condition; the definition is vacuous when $2k \geq n$.

Since the tensors $T_{2k}$ are divergence‑free, any $(2k)$-Einstein metric has constant Gauss--Bonnet curvature $h_{2k}$ (see \cite[Proposition 3.4]{Labbi-variational}). Recall that on a hypersurface in Euclidean space, $h_{2k}$ coincides (up to a constant factor) with the $2k$-th mean curvature. Using a result of Ros and Korevaar \cite{Ros-Kor}, we immediately obtain:

\begin{corollary}
The only compact embedded hypersurface in $\mathbb{R}^{n+1}$ that is $(2k)$-Einstein for some $k$ with $2 \leq 2k < n$ is the round sphere.
\end{corollary}

\begin{remark}
On an irreducible locally symmetric space the tensor $T_{2k}$ is symmetric and parallel, hence proportional to the metric. Thus such spaces are $(2k)$-Einstein for every admissible $k$.
\end{remark}

The next proposition characterizes $(2k)$-Einstein metrics in terms of the orthogonal decomposition of $R^k$.

\begin{proposition}[\cite{Labbi-variational}]\label{prop:2k-Einstein-char}
Let $(M,g)$ be a Riemannian manifold of dimension $n$.
\begin{itemize}
\item If $n \geq 2k+1$, then $g$ is $(2k)$-Einstein if and only if the component $\omega_1$ in the orthogonal decomposition
  \[
  R^k = \omega_{2k} + g\omega_{2k-1} + \dots + g^{2k-1}\omega_1 + g^{2k}\omega_0
  \]
  vanishes (i.e., $\omega_1 = 0$).
\item If $n = 2k+1$ (odd dimension), then $g$ is $(2k)$-Einstein if and only if it has constant $2k$-sectional curvature.
\end{itemize}
\end{proposition}

\begin{proof}
For $n \geq 4k$ the first statement follows directly from Corollary~\ref{cor:vanishing-omega_i}. When $2k < n < 4k$, we have $R^k = g^{4k-n}A$ for some $(n-2k,n-2k)$ double form $A$ (see the discussion before Corollary~\ref{cor:vanishing-omega_i}). In this case the condition $\cc^{2k-1}R^k = \text{constant}\cdot g$ is equivalent to $\cc^{\,n-2k-1}A = \text{constant}\cdot g$, and applying Corollary~\ref{cor:vanishing-omega_i} to $A$ yields $\omega_1 = 0$.

For the second part, note that when $n=2k+1$ the decomposition reduces to $R^k = g^{2k-1}(\omega_1 + g\omega_0)$. Hence $\omega_1 = 0$ exactly when $R^k$ is a constant multiple of $g^{2k-1}$, i.e., when $g$ has constant $2k$-sectional curvature.
\end{proof}

The following result connects the classical Einstein condition with the $L^2$-norm of the full curvature tensor.

\begin{proposition}\label{prop:2-and-4-Einstein}
In dimension $n > 4$, a Riemannian metric that is simultaneously $2$-Einstein and $4$-Einstein is a critical point of the functional
\[
G_2(g) = \int_M \|R\|^2\,\mathrm{dvol}_g
\]
when restricted to unit‑volume metrics.
\end{proposition}

\begin{proof}
Using Lemma~\ref{lemma:diff-metric} we compute the gradient of $G_2$:
\begin{align*}
G'_2.h &= -\bigl\langle F_h(R), R\bigr\rangle
        + 2\bigl\langle R'h, R\bigr\rangle
        + \tfrac12\bigl\langle \|R\|^2 g, h\bigr\rangle \\
      &= -2\bigl\langle gh, R\circ R\bigr\rangle
         + \tfrac12\bigl\langle F_h(R), R\bigr\rangle
         - \tfrac12\bigl\langle (D\tilde D+\tilde D D)(h), R\bigr\rangle
         + \tfrac12\bigl\langle \|R\|^2 g, h\bigr\rangle \\
    &=-2\langle c(R\circ R),h\rangle+\langle c(R\circ R),h\rangle-\frac{1}{2}\langle  \{D\tilde D+\tilde DD\}(h),R\rangle+\frac{1}{2}\langle g||R||^2,h\rangle\\
      &= \bigl\langle \tfrac12\|R\|^2 g - \cc(R\circ R)
         - \tfrac12(\delta\tilde\delta+\tilde\delta\delta)(R), h\bigr\rangle .
\end{align*}

For an Einstein metric the Riemann tensor is harmonic, so the divergence term vanishes. Moreover, identity \eqref{eq:R-identity} shows that $\cc(R\circ R)$ is proportional to the metric and equation \eqref{eq:Avez} shows that $||R||^2$ is constant. Restricting to variations that preserve unit volume ($\langle g,h\rangle = 0$) finishes the proof.
\end{proof}

\section{Thorpe and anti-Thorpe manifolds}

\subsection{Thorpe and anti-Thorpe manifolds in dimension $4k$}

Recall that a four‑dimensional Riemannian manifold is Einstein precisely when its curvature tensor satisfies the self‑duality condition $\ast R = R$; the anti‑self‑duality condition $\ast R = -R$ characterizes conformally flat metrics with zero scalar curvature. In this section we study analogous conditions for the higher exterior powers $R^k$ in dimension $n=4k$.

\begin{definition}\label{def:Thorpe-4k}
A Riemannian manifold $(M,g)$ of dimension $n=4k$ ($k \geq 1$) is called
\begin{enumerate}
\item a \emph{Thorpe manifold} if $\ast R^k = R^k$,
\item an \emph{anti-Thorpe manifold} if $\ast R^k = -R^k$.
\end{enumerate}
\end{definition}

Flat manifolds are both Thorpe and anti-Thorpe; manifolds of constant curvature are Thorpe. A non‑trivial family of examples is provided by products of space forms.

\begin{proposition}[\cite{Kim}]
The Riemannian product $\mathbb{S}^{2k} \times \mathbb{H}^{2k}$ is Thorpe if $k$ is even and anti-Thorpe if $k$ is odd.
\end{proposition}

\begin{proof}
Both $\ast R^k$ and $R^k$ satisfy the first Bianchi identity, so it suffices to compare their sectional curvatures. For a $2k$-plane $P$ tangent to the $\mathbb{S}^{2k}$ factor the sectional curvature of $R^k$ equals $\frac{(2k)!}{2^k}$; for a plane tangent to the $\mathbb{H}^{2k}$ factor it equals $\frac{(-1)^k(2k)!}{2^k}$. All other planes have zero curvature. The sectional curvature of $\ast R^k$ at $P$ is the sectional curvature of $R$ at the orthogonal complement of $P$. The stated parity condition follows immediately.
\end{proof}

Thorpe \cite{Thorpe} proved an important topological obstruction for such manifolds:

\begin{theorem}[\cite{Thorpe}]
Let $(M,g)$ be a compact orientable $4k$-dimensional Thorpe manifold. Then the Euler characteristic satisfies
\[
\chi(M) \geq \frac{(k!)^2}{(2k)!}\,|p_k(M)|,
\]
where $p_k(M)$ is the $k$-th Pontrjagin number of $M$.
\end{theorem}
In particular, a manifold with $\chi(M)<0$ cannot admit a Thorpe metric.

The next result shows that Thorpe and anti-Thorpe metrics are absolute minimisers of the $L^2$-norm of $R^k$ when the Euler characteristic does not vanish.

\begin{proposition}\label{prop:L2-minimum}
Let $M$ be a compact orientable $4k$-manifold and consider the functional
\[
\|R^k\|_{L^2}(g) = \int_M \|R^k\|^2\,\mathrm{dvol}_g .
\]
\begin{enumerate}
\item If $\chi(M) > 0$, the functional attains its absolute minimum exactly on the Thorpe metrics, and the minimum value is $(2\pi)^{2k}(2k)!\,\chi(M)$.
\item If $\chi(M) < 0$, the functional attains its absolute minimum exactly on the anti-Thorpe metrics, and the minimum value is $-(2\pi)^{2k}(2k)!\,\chi(M)$.
\end{enumerate}
\end{proposition}

\begin{proof}
Since $\ast$ is an isometry,
\[
\|R^k \pm \ast R^k\|^2 = 2\|R^k\|^2 \pm 2\langle R^k, \ast R^k\rangle .
\]
The Gauss--Bonnet integrand $h_{4k}$ satisfies (see \cite{Labbi-JAUMS})
\[
h_{4k} = \ast R^{2k} = \ast(R^k R^k) = \langle R^k, \ast R^k\rangle .
\]
Hence $\|R^k\|^2 \pm h_{4k} = \tfrac12\|R^k \pm \ast R^k\|^2 \geq 0$, with equality exactly when $\ast R^k = \pm R^k$. Integrating over $M$ and using the Gauss--Bonnet theorem gives the desired inequality.
\end{proof}

\begin{remarks}
\begin{enumerate}
    \item The previous proposition generalizes \cite[Proposition 4.82]{Besse} obtained for $k=1$
\item When $\chi(M)=0$, Thorpe metrics need not be absolute minimisers. For example, the product $T^{4k}=T^2\times T^{4k-2}$ with a non‑flat metric on $T^2$ and a flat metric on $T^{4k-2}$ is Thorpe but not flat.
\end{enumerate}
\end{remarks}

The duality condition can be translated into a simple condition on the components of the orthogonal decomposition of $R^k$.

\begin{proposition}\label{prop:Thorpe-components}
On a $4k$-dimensional Riemannian manifold let  $R^k = \sum_{i=0}^{2k} g^{2k-i}\omega_i$ be the orthogonal decomposition of $R^k$. Then
\begin{enumerate}
\item $(M,g)$ is Thorpe if and only if all odd components $\omega_1,\omega_3,\dots,\omega_{2k-1}$ vanish;
\item $(M,g)$ is anti-Thorpe if and only if all even components $\omega_0,\omega_2,\dots,\omega_{2k}$ vanish.
\end{enumerate}
\end{proposition}

\begin{proof}
From \cite[Theorem 4.3]{Labbi-double-forms} we have $\ast R^k = \sum_{i=0}^{2k} (-1)^i g^{2k-i}\omega_i$. Therefore
\[
R^k \pm \ast R^k = \sum_{i=0}^{2k} (1 \pm (-1)^i) g^{2k-i}\omega_i .
\]
The result follows from the uniqueness of the orthogonal decomposition.
\end{proof}

\subsection{Generalized Thorpe conditions in even dimensions}

Theorem~5.8 of \cite{Labbi-double-forms} motivates the following extension of the definition to all even dimensions.

\begin{theorem}[{\cite[Theorem 5.8]{Labbi-double-forms}}]\label{thm:general-duality}
Let $(M,g)$ be a Riemannian manifold of even dimension $n=2r$ and let $k$ be an integer with $2k \leq r$. Write $R^k = \sum_{i=0}^{2k} g^{2k-i}\omega_i$. Then
\begin{enumerate}
\item $\ast\!\left(\frac{g^{r-2k}}{(r-2k)!}R^k\right) = \frac{g^{r-2k}}{(r-2k)!}R^k$ if and only if all odd components $\omega_1,\omega_3,\dots,\omega_{2k-1}$ vanish;
\item $\ast\!\left(\frac{g^{r-2k}}{(r-2k)!}R^k\right) = -\frac{g^{r-2k}}{(r-2k)!}R^k$ if and only if all even components $\omega_0,\omega_2,\dots,\omega_{2k}$ vanish.
\end{enumerate}
\end{theorem}

\begin{definition}\label{def:2k-Thorpe}
Let $(M,g)$ be a Riemannian manifold of even dimension $n=2r$ and let $k$ be an integer with $2 \leq 2k \leq r$. We say that $(M,g)$ is
\begin{enumerate}
\item \emph{$(2k)$-Thorpe} if $\ast\!\left(g^{r-2k}R^k\right) = g^{r-2k}R^k$;
\item \emph{$(2k)$-anti-Thorpe} if $\ast\!\left(g^{r-2k}R^k\right) = -g^{r-2k}R^k$.
\end{enumerate}
\end{definition}

For $k=1$ we recover Einstein metrics ($2$-Thorpe) and conformally flat metrics with zero scalar curvature ($2$-anti-Thorpe). When $r=2k$ ($n=4k$) the definition coincides with Definition~\ref{def:Thorpe-4k}.

\begin{examples}\label{ex:Thorpe-examples}
\begin{enumerate}
\item Space forms of dimension $2r$ are $(2k)$-Thorpe for every $k$ with $2 \leq 2k \leq r$.
\item Any manifold of even dimension for which the tensor $R^k$ has constant sectional curvature is $(2k)$-Thorpe.
\item Flat manifolds (or manifolds with $R^k=0$) are both $(2k)$-Thorpe and $(2k)$-anti-Thorpe.
\item The Riemannian product of a space form of dimension $r$ with itself is $(2k)$-Thorpe for all admissible $k$.
\item The torus $T^{2r}=T^2\times T^{2r-2}$ with a non‑flat metric on $T^2$ and a flat metric on $T^{2r-2}$ satisfies $R^k=0$ for all $k \geq 2$, hence it is both $(2k)$-Thorpe and $(2k)$-anti-Thorpe.
\end{enumerate}
\end{examples}

\begin{proposition}\label{prop:Thorpe-implications}
\begin{itemize}
\item Every $(2k)$-Thorpe manifold is $(2k)$-Einstein; in particular, its Gauss--Bonnet curvature $h_{2k}$ is constant if $n>2k$.
\item Every $(2k)$-anti-Thorpe manifold has vanishing Gauss--Bonnet curvature: $h_{2k}=0$.
\end{itemize}
\end{proposition}

\begin{proof}
Both statements follow directly from Theorem~\ref{thm:general-duality} and the relation between the components $\omega_i$ and the contractions of $R^k$.
\end{proof}

\begin{corollary}\label{cor:hypersurface-obstruction}
For $n \geq 4k$ there are no compact embedded hypersurfaces in $\mathbb{R}^{n+1}$ that are $(2k)$-anti-Thorpe. Moreover, the round sphere is the only compact embedded $(2k)$-Thorpe hypersurface in $\mathbb{R}^{n+1}$.
\end{corollary}

\begin{proof}
On a hypersurface, $h_{2k}$ coincides (up to a constant factor) with the $2k$-th mean curvature $H_{2k}$. The result follows from Proposition~\ref{prop:Thorpe-implications} and the characterisation of hypersurfaces with constant $H_{2k}$ by Ros and Korevaar \cite{Ros-Kor}.
\end{proof}

The following theorem gives a geometric obstruction for the existence of Thorpe or anti-Thorpe metrics.

\begin{theorem}[{\cite[Theorem 6.7]{Labbi-double-forms}}]\label{thm:obstruction}
Let $(M,g)$ be a $(2k)$-Thorpe (respectively $(2k)$-anti-Thorpe) manifold of even dimension $n \geq 4k \geq 4$. Then the $4k$-th Gauss--Bonnet curvature satisfies
\[
h_{4k} \geq 0 \quad (\text{respectively } h_{4k} \leq 0),
\]
with equality if and only if $R^k = 0$ (i.e., the manifold is $k$-flat).
\end{theorem}

\begin{proof}
Write $n=2r$. Using $\ast(g^{r-2k}R^k) = \pm g^{r-2k}R^k$, we compute
\[
(n-4)!\,h_{4k} = \ast\!\left(g^{n-4k}R^{2k}\right)
               = \ast\!\left((g^{r-2k}R^k)(g^{r-2k}R^k)\right)
               = \pm \ast\!\left(\ast(g^{r-2k}R^k)\,g^{r-2k}R^k\right)
               = \pm \|g^{r-2k}R^k\|^2 .
\]
The right‑hand side is non‑negative (resp. non‑positive) and vanishes exactly when $g^{r-2k}R^k=0$, which by Proposition~\ref{cor:cancellation} is equivalent to $R^k=0$.
\end{proof}

The duality condition can also be expressed directly through contractions of $R^k$.

\begin{proposition}\label{prop:contraction-criterion}
Let $(M,g)$ be a Riemannian manifold of even dimension $n=2r$ and let $2 \leq 2k \leq r$. Then
\begin{enumerate}
\item $(M,g)$ is $(2k)$-Thorpe if and only if
  \[
  \sum_{s=1}^{2k} \frac{(-1)^s}{s!}\,
      \frac{g^{s-1}\cc^s R^k}{(r-2k+s)!} = 0 .
  \]
\item $(M,g)$ is $(2k)$-anti-Thorpe if and only if
  \[
  \frac{2R^k}{(r-2k)!}
    = \sum_{s=1}^{2k} \frac{(-1)^{s-1}}{s!}\,
      \frac{g^{s}\cc^s R^k}{(r-2k+s)!} = 0 .
  \]
\end{enumerate}
\end{proposition}

\begin{proof}
Formula (15) in \cite{Labbi-double-forms} shows that
$$  \str \Bigl( \frac{g^{r-2k}R^k}{(r-2k)!}\Bigr)= g^{r-2k}\sum_{s=0}^{2k}\frac{(-1)^s}{s!}\frac{g^{s}\cc^s R^k}{(r-2k+s)!}=\frac{g^{r-2k}R^k}{(r-2k)!}+g^{r-2k}\sum_{s=1}^{2k}\frac{(-1)^s}{s!}\frac{g^{s}\cc^s R^k}{(r-2k+s)!}.$$
Since $r-2k\leq n-4k$, it follows from the above identity and  from Proposition~\ref{cor:cancellation} that $2k$-Thorpe condition is equivalent to 
$$\sum_{s=1}^{2k}\frac{(-1)^s}{s!}\frac{g^{s}\cc^s R^k}{(r-2k+s)!}=0.$$
One more application of Proposition~\ref{cor:cancellation} yields the required result as $1\leq n-2(2k-1)$.\\
The proof of the second part is completely similar.
\end{proof}

\begin{remark}
The formulas in Proposition~\ref{prop:contraction-criterion} can be used to extend the notions of $(2k)$-Thorpe and $(2k)$-anti-Thorpe to arbitrary dimensions $n > 2k$. For example, for $n>4$ one can define a metric to be $4$-Thorpe if
\[
\cc R^2 = \frac{g\,\cc^2R^2}{n-4}
          - \frac{2g^2\cc^3R^2}{3(n-2)(n-4)}
          + \frac{g^3\cc^4R^2}{3n(n-2)(n-4)} .
\]
\end{remark}

\subsection{Harmonicity of Thorpe tensors}

An important feature of $(2k)$-Thorpe and $(2k)$-anti-Thorpe metrics is that the tensor $R^k$ becomes harmonic.

\begin{proposition}\label{prop:Thorpe-harmonic}
Let $(M,g)$ be a Riemannian manifold of even dimension $n=2r$ and let $2 \leq 2k \leq r$. If $(M,g)$ is $(2k)$-Thorpe (resp. $(2k)$-anti-Thorpe), then $R^k$ is a harmonic double form.
\end{proposition}

\begin{proof}
Apply the operator $\ast\delta$ to the equation $\ast(g^{r-2k}R^k) = \pm g^{r-2k}R^k$. Using the identity $\ast\delta\ast = \pm D$ we obtain
\[
\pm \ast\delta(g^{r-2k}R^k) = \ast\delta\ast(g^{r-2k}R^k)
                           = \pm D(g^{r-2k}R^k) = 0 .
\]
Hence $\delta(g^{r-2k}R^k)=0$. Formula (46) of \cite{Labbi-Lapl-Lich} implies $g^{r-2k}\delta R^k=0$, and by the cancellation property (Proposition~\ref{cor:cancellation}) we conclude $\delta R^k=0$ (note that $r-2k+2k-1+2k < 2r+1$). Since $R^k$ is closed by the second Bianchi identity, it is harmonic.
\end{proof}

Harmonicity together with positivity assumptions on the curvature yields strong rigidity results.

\begin{theorem}[{\cite[Theorem B]{Labbi-Lapl-Lich}}]\label{thm:Rk-harmonic-rigidity}
Let $(M,g)$ be a closed connected Riemannian manifold of dimension $n \geq 4k \geq 4$ such that $R^k$ is a harmonic $(2k,2k)$ double form. Let $j$ be an integer with $1 \leq j \leq 2k$. If the curvature operator $R$ is $\bigl\lfloor\frac{n-j+1}{2}\bigr\rfloor$-positive, then $\cc^{2k-j}R^k = \lambda g^j$ for some constant $\lambda$. In particular,
\begin{enumerate}
\item if $R$ is $\bigl\lfloor\frac{n-2k+1}{2}\bigr\rfloor$-positive, then $R^k$ has constant sectional curvature;
\item if $R$ is $\bigl\lfloor\frac{n}{2}\bigr\rfloor$-positive, then $(M,g)$ is $(2k)$-Einstein;
\item if $R$ is $\bigl\lfloor\frac{n-2k+2}{2}\bigr\rfloor$-positive, then $(M,g)$ is hyper $(2k)$-Einstein.
\end{enumerate}
\end{theorem}

\begin{corollary}\label{cor:Thorpe-rigidity}
Let $(M,g)$ be a closed connected $(2k)$-Thorpe manifold of dimension $n \geq 4k \geq 4$.
\begin{enumerate}
\item If $R$ is $\bigl\lfloor\frac{n-2k+1}{2}\bigr\rfloor$-positive, then $R^k$ has constant sectional curvature.
\item If $R$ is $\bigl\lfloor\frac{n-2k+2}{2}\bigr\rfloor$-positive, then $(M,g)$ is hyper $(2k)$-Einstein.
\end{enumerate}
\end{corollary}

\begin{corollary}\label{cor:anti-Thorpe-rigidity}
Let $(M,g)$ be a closed connected $(2k)$-anti-Thorpe manifold of dimension $n \geq 4k \geq 4$. If $R$ is $\bigl\lfloor\frac{n-2k+2}{2}\bigr\rfloor$-positive, then $R^k=0$ (i.e., $(M,g)$ is $k$-flat).
\end{corollary}

\subsection{Locally conformally flat case}

For locally conformally flat manifolds, the Thorpe conditions can be expressed purely in terms of the Schouten tensor $A$ (recall that in this case $R = gA$).

\begin{proposition}\label{prop:conf-flat-Thorpe}
Let $(M,g)$ be a locally conformally flat manifold of even dimension $n=2r \geq 4k \geq 4$ with Schouten tensor $A$. Then $(M,g)$ is $(2k)$-Thorpe with $k$ even (resp. $(2k)$-anti-Thorpe with $k$ odd) if and only if
\begin{equation}\label{eq:conf-flat-condition}
\sum_{s=1}^{k} (-1)^s \,
    \frac{g^{s-1}\cc^s A^k}{s!(r-k+s)!} = 0 .
\end{equation}
\end{proposition}

\begin{proof}
Because $R=gA$, we have $R^k = g^k A^k$ and $g^{r-2k}R^k = g^{r-k}A^k$. The duality condition $\ast(g^{r-k}A^k) = \pm g^{r-k}A^k$ together with formula (15) of \cite{Labbi-double-forms} yields
\[
(-1)^k\frac{g^{r-k}A^k}{(r-k)!}
   + \sum_{s=1}^{k} (-1)^{s+k}
      \frac{g^{r-k+s}\cc^s A^k}{s!(r-k+s)!}
   = \pm \frac{g^{r-k}A^k}{(r-k)!} .
\]
Cancellation property (Proposition~\ref{cor:cancellation}) gives \eqref{eq:conf-flat-condition}.
\end{proof}

Specialising to $k=2$ and $k=3$ we obtain more explicit criteria.

\begin{corollary}\label{cor:explicit-criteria}
Let $(M,g)$ be a locally conformally flat manifold of even dimension $n$.
\begin{enumerate}
\item For $n \geq 8$, $(M,g)$ is $4$-Thorpe if and only if its Schouten tensor satisfies
  \begin{equation}\label{eq:4-Thorpe-conf}
  A\circ A - (\cc A)A = \frac{\|A\|^2 - (\cc A)^2}{n}\,g .
  \end{equation}
\item For $n \geq 12$ and $\cc A = 0$, $(M,g)$ is $6$-anti-Thorpe if and only if
  \begin{equation}\label{eq:6-anti-Thorpe-conf}
  3n(n-2)(A\circ A)A
    - 3n\bigl(\|A\|^2 A - 2A\circ A\circ A\bigr)g
    - 4 g^2 \cc(A\circ A\circ A)= 0 .
  \end{equation}
\end{enumerate}
\end{corollary}

\begin{proof}
The identities follow from Proposition~\ref{prop:conf-flat-Thorpe} together with the Greub–Vanstone formulas (see \cite[Proposition 6.1]{Labbi-Bel}):
\begin{align*}
\cc(A^2) &= 2(\cc A)A - 2A\circ A, &
\cc^2(A^2) &= 2(\cc A)^2 - 2\|A\|^2, \\
\cc(A^3) &= -6(A\circ A)A, &
\cc^2(A^3) &= -6\|A\|^2 A + 12A\circ A\circ A, \\
\cc^3(A^3) &= 12\cc(A\circ A\circ A) .
\end{align*}
\end{proof}

\begin{corollary}\label{cor:SxH-examples}
Let $\mathbb{S}^r(c)$ and $\mathbb{H}^r(-c)$, $c>0$,  denote the space forms of constant curvature $c$ and $-c$, respectively.
\begin{enumerate}
\item The product $\mathbb{S}^r(c) \times \mathbb{H}^r(-c)$ is $4$-Thorpe but not $4$-anti-Thorpe for every $r \geq 4$.
\item The product $\mathbb{S}^r(c) \times \mathbb{H}^r(-c)$ is both $4$-Thorpe and $6$-anti-Thorpe for every $r \geq 6$.
\end{enumerate}
\end{corollary}

\begin{proof}
For $\mathbb{S}^r(c) \times \mathbb{H}^r(-c)$ one computes $\cc A = 0$, $A\circ A = \frac{c^2}{4} g$, $\|A\|^2 = \tfrac{rc^2}{2}$, and $A\circ A\circ A = \frac{c^2}{4} A$. Substituting these values into \eqref{eq:4-Thorpe-conf} and \eqref{eq:6-anti-Thorpe-conf} verifies the Thorpe and anti-Thorpe conditions. To show that it is not $4$-anti-Thorpe, remark that Gauss-Bonnet curvature $h_4$, which is in our case a constant multiple of the $\sigma_2$ curvature,  is not zero.
\end{proof}

For conformally flat $4$-Thorpe manifolds we obtain a complete classification.

\begin{proposition}\label{prop:classif-conf-flat-4Thorpe}
Let $(M,g)$ be a locally conformally flat manifold of even dimension $n = 2r \geq 8$. The following statements are equivalent:
\begin{enumerate}
\item $(M,g)$ is $4$-Thorpe;
\item $(M,g)$ is $4$-Einstein;
\item $(M,g)$ satisfies $\iota_{\Ric}R = \frac{\|\Ric\|^2}{n}\,g$;
\item $(M,g)$ is either flat or covered by one of $\mathbb{S}^n$, $\mathbb{H}^n$, or $\mathbb{S}^r(c)\times\mathbb{H}^r(-c)$.
\end{enumerate}
\end{proposition}

\begin{proof}
First, Proposition \ref{prop:conf-flat-Thorpe} shows that the $4$-Thorpe condition is equivalent to 
\begin{equation}\label{4-Thorpe}
\cc A^2=\frac{\cc^2 A^2}{n}g.
\end{equation}
The $4$-Einstein condition  is defined by requiring
\begin{equation}\label{4-Einstein} \cc^3 (g^2A^2)=\frac{\cc^4(g^2 A^2)}{n}g.\end{equation}
Formula (6) in \cite{Labbi-double-forms} shows that
$$\cc^3 (g^2A^2)=6(n-3)(\cc^2A^2)g+6(n-3)(n-4)\cc A^2,\,\,\,\cc^4 (g^2A^2)=12(n-3)(n-2)\cc^2A^2.$$
Plugging the above two quantities in equation \ref{4-Einstein}, we see that the first two properties are equivalent.\\
Next, the $\iota_{\Ric}R$-Einstein condition is defined by requiring
\begin{equation}\label{Weakly-Einstein} \iota_{\Ric}R=\frac{||\Ric||^2}{n}g.\end{equation}
Recall that $\Ric=g\cc A+(n-2)A$, a direct substitution shows that equation \ref{Weakly-Einstein} is equivalent to
\begin{equation}\label{Weakly-Einstein2} \iota_{A}R+(\cc A)A=\frac{2(\cc A)^2+(n-2)||A||^2}{n}g.\end{equation}
Greub-Vanstone  identities shows that $\iota_{A}R=\iota_{A}(gA)=||A||^2+\cc A^2-(\cc A)A.$ Consequently, equation \ref{Weakly-Einstein2} takes the form
$$\cc A^2=\frac{2(\cc A)^2-2||A||^2}{n}g,$$
which is precisely equation \ref{4-Einstein}, as $\cc^2 A^2= 2(\cc A)^2-2||A||^2.$\\
Finally, the Schouten  tensor $A$ of a  locally conformally flat manifold of dimension $\geq 4$ is a Codazzi tensor. This can be justified quickly, as in this case $0=DR=gDA$ and by the cancellation property \ref{cor:cancellation} one has $DA=0$ for $n\geq 4$. On the other hand, the quadratic equation \ref{eq:4-Thorpe-conf} shows that for a $4$-Thorpe conformally flat manifold, the tensor  $A$ has at most two distinct eigenvalues, say $\lambda_1,\lambda_2$, such that $\lambda_1+\lambda_2=\cc A$ and $p\lambda_1+(n-p)\lambda_2=\cc A$, where $p$ is the multiplicity of $\lambda_1$. Marino-Villar \cite{Marino},  used the above facts and results by Derdzinski \cite{Derdzinski} and Merton \cite{Merton} to conclude that the manifold is either Einstein or locally is the product $\mathbb{S}^r(c) \times \mathbb{H}^r(-c)$, see the proof of Theorem 2 in \cite{Marino}. 
\end{proof}
\subsection{Extending the definition of $(2k)$-Thorpe and $(2k)$-anti-Thorpe to lower dimensions}
In this sub-section we assume that $2k\leq n=2r<4k$. Proposition 2.1 in \cite{Labbi-Balkan} shows that 
$$R^k=g^{4k-n}A,$$
with $A$ a symmetric $(n-2k,n-2k)$ double form. \\
The  $2k$-Thorpe condition (resp. anti-Thorpe) in this case reads
\begin{equation}
\ast \bigl(g^{r-2k}g^{4k-n}A\bigr)=g^{r-2k}g^{4k-n}A, \,\,\, \Bigl({\rm resp.}\,\,\,\, \ast  \bigl(g^{r-2k}g^{4k-n}A\bigr)=-g^{r-2k}g^{4k-n}A\Bigr).
\end{equation}
Or equivalently,
\begin{equation}\label{new-eq}
\ast \bigl(g^{2k-r}A\bigr)=g^{2k-r}A, \,\,\, \Bigl({\rm resp.}\,\,\,\, \ast \bigl(g^{2k-r}A\bigr)=-g^{2k-r}A\Bigr).
\end{equation}
 
 Let now  $A= \sum_{i=0}^{n-2k}g^i\omega_{n-2k-i}$ be the orthogonal decomposition of the double form $A$, it follows that  $R^k= \sum_{i=0}^{n-2k}g^{2k-i}\omega_{i}$. In particular, we remark the vanishing of $\omega_i=0$ for $n-2k<i\leq 2k$ are obtained for free because of the dimension restriction.\\
By imitating the proof of Theorem \ref{thm:general-duality}, one can see  that  
\begin{equation*}
\begin{split}
\ast \bigl(g^{2k-r}A\bigr)=& g^{2k-r}A \iff \omega_i=0,\,\, {\rm for}\,\, i\,\, {\rm odd}\,\, 0\leq i\leq n-2k.\\
\ast \bigl(g^{2k-r}A\bigr)=& -g^{2k-r}A \iff \omega_i=0,\,\, {\rm for}\,\, i\,\, {\rm even}\,\, 0\leq i\leq n-2k.
\end{split}
\end{equation*}

In the extreme case where $n=2k$, $A$ is a $(0,0)$ double form, that is a scalar function. Precisely, it  is a constant scalar multiple of the Gauss-Bonnet curvature $h_{2k}$. 
The $2k$-Thorpe condition here is vacuous. While the $2k$-anti-Thorpe condition is equivalent to the vanishing of $\omega_0$, that is the vanishing of the $h_{2k}$ Gauss-Bonnet curvature. 
We summarize the previous discussion in the following proposition
\begin{proposition}\label{lower-dimensions}
Let $(M,g)$ be a Riemannian manifold of even dimension $n=2r$ and let $k$ be an integer with $2k < n<4k$. Let $R^k = \sum_{i=0}^{2k} g^{2k-i}\omega_i$ be the orthogonal decomposition of $R^k$. Then $\omega_i=0$ for $n-2k<i\leq 2k$, furthermore 
\begin{enumerate}
\item  $(M,g)$ is $(2k)$-Thorpe if and only if all odd components $\omega_1,\omega_3,\dots,\omega_{(n-2k-1)}$ vanish;
\item $(M,g)$ is $(2k)$-anti-Thorpe if and only if all even components $\omega_0,\omega_2,\dots,\omega_{n-2k}$ vanish.
\end{enumerate}    
\end{proposition}

In particular, for $n=2k+2$, being $(2k)$-Thorpe is equivalent to being $(2k)$-Einstein. For instance, in dimension $6$, a metric is $4$-Thorpe if and only if it is $4$-Einstein.

\subsection{Extending the definition of $(2k)$-Thorpe and $(2k)$-anti-Thorpe to odd dimensions}
The definitions and results of the previous sections are formulated for even-dimensional manifolds, taking advantage of the natural self-duality structures available when $n=2r$. However, the algebraic conditions expressed in terms of the orthogonal decomposition or in terms of contractions make sense for \emph{all} dimensions $n>2k$. This subsection explains how to extend the notions of $(2k)$-Thorpe and $(2k)$-anti-Thorpe to odd-dimensional Riemannian manifolds.

Recall that for a $(p,p)$ double form $\omega$ on an $n$-dimensional space, the orthogonal decomposition $\omega = \sum_{i=0}^{p} g^{p-i}\omega_i$ with trace‑free $\omega_i$ is well‑defined regardless of the parity of $n$ (Proposition~\ref{prop:orth-decomp}). Moreover, Corollary~\ref{cor:vanishing-omega_i} characterises the vanishing of a component $\omega_{p-k}$ by the condition $\cc^k\omega = gA$ for some double form $A$.

\begin{definition}[$(2k)$-Thorpe and $(2k)$-anti-Thorpe in arbitrary dimension]\label{def:odd-dim-Thorpe}
Let $(M,g)$ be a Riemannian manifold of dimension $n > 2k$ ($k \geq 1$). Denote by $\omega_0,\omega_1,\dots,\omega_{2k}$ the trace‑free components in the orthogonal decomposition $R^k = \sum_i g^{2k-i}\omega_i$.
\begin{enumerate}
\item $(M,g)$ is called \emph{$(2k)$-Thorpe} if $\omega_i = 0$ for every odd index $i$ with $0 \leq i \leq 2k$.
\item $(M,g)$ is called \emph{$(2k)$-anti-Thorpe} if $\omega_i = 0$ for every even index $i$ with $0 \leq i \leq 2k$.
\end{enumerate}
\end{definition}

Equivalently, using Corollary~\ref{cor:vanishing-omega_i}, these conditions can be stated purely in terms of contractions.

\begin{proposition}[Contraction formulation]\label{prop:odd-dim-contraction}
Let $(M,g)$ be a Riemannian manifold of dimension $n > 2k$.
\begin{enumerate}
\item $(M,g)$ is $(2k)$-Thorpe if and only if for every odd integer $i$ with $0 \leq i \leq 2k$ there exists a $(2k-i-1,2k-i-1)$ double form $A_i$ such that $\cc^i R^k = g A_i$.
\item $(M,g)$ is $(2k)$-anti-Thorpe if and only if for every even integer $i$ with $0 \leq i \leq 2k$ there exists a $(2k-i-1,2k-i-1)$ double form $A_i$ such that $\cc^i R^k = g A_i$.
\end{enumerate}
\end{proposition}

\begin{remark}
When $n$ is even, Definition~\ref{def:odd-dim-Thorpe} coincides exactly with Definition~\ref{def:2k-Thorpe} (see Theorem~\ref{thm:general-duality}). For $k=1$, the $(2)$-Thorpe condition reduces to $\omega_1=0$, i.e., $\cc R = \lambda g$ – the usual Einstein condition – while the $(2)$-anti-Thorpe condition reduces to $\omega_2=\omega_0=0$, i.e., $R=gA$ and $\cc^2 R = 0$, which on a manifold of dimension $n>3$ characterises conformally flat metrics with vanishing scalar curvature.

Throughout the paper we have chosen to work primarily in even dimensions because the Hodge star operator provides a natural self‑duality framework, and many results such as the harmonicity property and  the topological/geometrical obstructions  rely on this structure. Nevertheless, several algebraic and variational statements of this paper remain valid for arbitrary dimension $n>2k$ when the definitions are extended as above.
\end{remark}

\begin{example}
Consider a Riemannian manifold $(M^{2k+1},g)$ of odd dimension $n=2k+1$. In this case the orthogonal decomposition of $R^k$ reads
\[
R^k = g^{2k-1}(\omega_1 + g\omega_0),
\]
In this case, $(M,g)$ is $(2k)$-Thorpe precisely when $\omega_1=0$,  that  is $(2k)$-Einstein. Which in turn is equivalent to $R^k$ having constant $(2k)$-sectional curvature. It is $(2k)$-anti-Thorpe precisely when $\omega_0=0$, i.e., when the Gauss-Bonnet curvature $h_{2k}=0$. For $k=1$ this recovers the familiar dichotomy: Einstein vs. scalar‑flat metrics in dimension three.
\end{example}
The extension to odd dimensions underlines the algebraic nature of the Thorpe conditions, which are essentially requirements on the pattern of vanishing of the trace‑free components $\omega_i$. While the geometric interpretation via self‑duality is available only when $n$ is even, the algebraic definition provides a unified viewpoint for all dimensions $n>2k$.

\section{Critical metrics for the $L^2$-norm of Thorpe tensors}

This final section is devoted to the variational analysis of the functional
\[
G_{2k}(g)=\int_M\|R^k\|^2\,\mathrm{dvol}_g=\int_M\langle R^k,R^k\rangle\,\mathrm{dvol}_g,
\]
where $R^k$ is the $(2k)$-th Thorpe tensor, i.e., the $k$-th exterior power of the Riemann curvature tensor viewed as a $(2,2)$ double form. We compute its gradient, characterize its critical points, and establish connections with the previously introduced classes of metrics.

\subsection{Gradient of the $L^2$-norm functional}

The following proposition gives an explicit formula for the gradient of $G_{2k}$. It is the starting point for all subsequent characterizations of critical metrics.

\begin{proposition}\label{prop:grad-G2k}
For $2\leq 2k\leq n$, the gradient of the functional $G_{2k}$ on the space of Riemannian metrics is
\[
\nabla G_{2k}= \frac12\|R^k\|^2g-\frac1{(2k-1)!}\cc^{2k-1}(R^k\circ R^k)
              -k\,\iota_{R^{k-1}}\!\bigl(\delta\tilde\delta R^k\bigr).
\]
(For $k=1$ we adopt the conventions $R^0=1$ and $\iota_1\omega=\omega$.)
\end{proposition}

\begin{proof}
Let $h$ be a symmetric $(1,1)$ double form. Using Lemma~\ref{lemma:diff-metric},
\begin{align*}
G_{2k}'.h &= -\bigl\langle F_h(R^k),R^k\bigr\rangle
           +2\bigl\langle kR^{k-1}R'h,R^k\bigr\rangle
           +\tfrac12\bigl\langle \|R^k\|^2g,h\bigr\rangle \\
          &= I + II + III .
\end{align*}

Because $R^k$ is symmetric, formula \eqref{Fh-formula} yields
\begin{align*}
I &= -\bigl\langle \tfrac{g^{2k-1}h}{(2k-1)!}\circ R^k
                  +R^k\circ\tfrac{g^{2k-1}h}{(2k-1)!},R^k\bigr\rangle \\
  &= -2\bigl\langle \tfrac{g^{2k-1}h}{(2k-1)!},R^k\circ R^k\bigr\rangle
   = -2\bigl\langle \tfrac1{(2k-1)!}\cc^{2k-1}(R^k\circ R^k),h\bigr\rangle .
\end{align*}

For the second term we use the expression for $R'h$ from Lemma~\ref{lemma:diff-curvature}:
\begin{align*}
II &= 2\bigl\langle kR^{k-1}R'h,R^k\bigr\rangle \\
   &= -\tfrac{k}{2}\bigl\langle R^{k-1}(D\tilde D+\tilde D D)(h),R^k\bigr\rangle
      +\tfrac{k}{2}\bigl\langle R^{k-1}F_h(R),R^k\bigr\rangle .
\end{align*}
Since $F_h$ acts as a derivation, $R^{k-1}F_h(R)=\frac1k F_h(R^k)$; moreover $\iota_{R^{k-1}}$ is the adjoint of left multiplication by $R^{k-1}$. Hence
\begin{align*}
II &= -\tfrac{k}{2}\bigl\langle (D\tilde D+\tilde D D)(h),\iota_{R^{k-1}}R^k\bigr\rangle
      +\tfrac12\bigl\langle F_h(R^k),R^k\bigr\rangle \\
   &= -\tfrac{k}{2}\bigl\langle h,(\delta\tilde\delta+\tilde\delta\delta)
                          \bigl(\iota_{R^{k-1}}R^k\bigr)\bigr\rangle -\tfrac12 I .
\end{align*}
Because $R^{k-1}$ is symmetric, $D\tilde D(R^{k-1}h)=\bigl(\tilde D D(R^{k-1}h)\bigr)^t$, and the operator $\delta\tilde\delta$ is the adjoint of $\tilde D D$ with respect to the integral scalar product (see \cite{Labbi-variational,Labbi-Lapl-Lich}). Consequently,
\[
II = -k\bigl\langle h,\iota_{R^{k-1}}\!\bigl(\delta\tilde\delta R^k\bigr)\bigr\rangle-\tfrac12 I .
\]

Adding $I$, $II$ and $III$ we obtain
\[
G_{2k}'.h =  I +II+ III
          = \bigl\langle \tfrac12\|R^k\|^2g
                        -\tfrac1{(2k-1)!}\cc^{2k-1}(R^k\circ R^k)
                        -k\,\iota_{R^{k-1}}\!\bigl(\delta\tilde\delta R^k\bigr),\,h\bigr\rangle ,
\]
which proves the stated formula.
\end{proof}

\begin{remarks}\label{rem:grad-remarks}
\begin{enumerate}
\item Using the identity $\tilde\delta = -\cc D - D\cc$ we have
  $\delta\tilde\delta R^k = -\delta D(\cc R^k)$. Therefore the gradient can also be written as
  \[
  \nabla G_{2k}= \frac12\|R^k\|^2g-\frac1{(2k-1)!}\cc^{2k-1}(R^k\circ R^k)
                +k\,\iota_{R^{k-1}}\!\bigl(\delta D(\cc R^k)\bigr).
  \]
  For $k=1$ this reduces to the well‑known expression, see \cite[Proposition 4.70]{Besse}
  \[
  \nabla G_2 = \frac12\|R\|^2g-\cc(R\circ R)+\delta D(\Ric).
  \]

\item A direct computation using the definition of the composition product shows that
 $$\cc(R\circ R)(u,v)=\sum_{i=1}^nR\circ R(u,e_i,v,e_i)=\sum_{i=1}^n\sum_{k<l}R(u,e_i,e_k,e_l)R(e_k,e_l,v,e_i)=2\check{R}(u,v).$$
\end{enumerate}
\end{remarks}

\subsection{Critical metrics under harmonicity assumptions}

When the Thorpe tensor $R^k$ is harmonic, the gradient simplifies considerably and we obtain a clean necessary and sufficient condition for criticality.

\begin{corollary}\label{cor:harmonic-critical}
Let $(M,g)$ be a Riemannian manifold of dimension $n\geq 4k$ such that $R^k$ is a harmonic double form. Then $g$ is a critical point of $G_{2k}$ restricted to unit‑volume metrics if and only if
\begin{equation}\label{eq:weak-Einstein-condition}
\frac1{(2k-1)!}\cc^{2k-1}(R^k\circ R^k)=\frac{2k}{n}\|R^k\|^2\,g .
\end{equation}
\end{corollary}

\begin{proof}
If $R^k$ is harmonic, $\delta R^k=0$ and consequently $\delta\tilde\delta R^k=0$. By Proposition~\ref{prop:grad-G2k},
\[
\nabla G_{2k}= \frac12\|R^k\|^2g-\frac1{(2k-1)!}\cc^{2k-1}(R^k\circ R^k).
\]
Restricting to variations that preserve the volume ($\langle g,h\rangle=0$), the metric is critical precisely when $\nabla G_{2k}$ is proportional to $g$. Contracting once the equation $\frac12\|R^k\|^2g-\frac1{(2k-1)!}\cc^{2k-1}(R^k\circ R^k)=\lambda g$ gives $\lambda=\frac{n-4k}{2n}\|R^k\|^2$, and substituting this value yields \eqref{eq:weak-Einstein-condition}.
\end{proof}

Recall that for $(2k)$-Thorpe and $(2k)$-anti-Thorpe metrics the tensor $R^k$ is harmonic (Proposition~\ref{prop:Thorpe-harmonic}). Hence the previous corollary immediately implies the following characterization.

\begin{proposition}\label{prop:critical-Thorpe}
\begin{enumerate}
\item If $n=4k$, every Thorpe and every anti-Thorpe metric is a critical point of $G_{2k}$ (without volume restriction).
\item If $n$ is even and $n>4k$, a $(2k)$-Thorpe (resp. $(2k)$-anti-Thorpe) metric is a critical point of $G_{2k}$ restricted to unit volume if and only if it satisfies \eqref{eq:weak-Einstein-condition}.
\end{enumerate}
\end{proposition}
\begin{remark}
    The previous proposition generalizes \cite[Corollary 4.72]{Besse} obtained for $k=1$.
\end{remark}
\begin{proof} 
First, we prove part 1). One can see that this part follows directly at once from Proposition \ref{prop:L2-minimum}. However, we are going to provide another proof which uses the above variational formula.   Let $n=4k$ and $ *R^k=\epsilon R^k$ with $\epsilon=\pm 1$.

\begin{equation*}
\begin{split}
2\langle \frac{1}{(2k-1)!}c^{2k-1}&(R^k\circ R^k),h\rangle=2\langle (R^k\circ R^k),\frac{1}{(2k-1)!}g^{2k-1}h\rangle \\
=& \langle R^k, R^k\circ \frac{1}{(2k-1)!}g^{2k-1}h+\frac{1}{(2k-1)!}g^{2k-1}h\circ R^k\rangle \\
=&\langle R^k,F_h(R^k)\rangle=\str\epsilon \bigl\{R^kF_h(R^k)\bigr\}=\epsilon/2\str (F_h(R^{2k}).
\end{split}
\end{equation*}
In the last step we used the fact that the operator $F_h$ operates by derivations on the exterior algebra of double forms. We continue the proof as follows
\begin{equation*}
\begin{split}
\str (F_h(R^{2k})=&\str\Bigl( \frac{g^{4k-1}h}{(4k-1)!}\circ R^{2k}+  R^{2k}\circ \frac{g^{4k-1}h}{(4k-1)!}\Bigr)\\
=&  
\str\bigl( \frac{g^{4k-1}h}{(4k-1)!}\bigr) \circ \str R^{2k}+ \str R^{2k}\circ \str \bigl(\frac{g^{4k-1}h}{(4k-1)!}\bigr).
\end{split}
\end{equation*}
We remark that 
 $*R^{2k}=*(R^kR^k)=\langle *R^k,R^k\rangle=\epsilon||R^k||^2$ and
 $$\str\frac{1}{(4k-1)!}g^{4k-1}h=\langle \str\frac{1}{(4k-1)!}g^{4k-1},h\rangle=\langle g,h\rangle.$$
 It follows that 
$\str (F_h(R^{2k})=2\epsilon||R^k||^2\langle g,h\rangle.$  We have therefore proved that for any symmetric bilinear form $h$ the identity
$$2\langle \frac{1}{(2k-1)!}c^{2k-1}(R^k\circ R^k),h\rangle=||R^k||^2\langle g,h\rangle.$$
This completes the proof of the first part.
The second part follows at once from the above corollary \ref{cor:harmonic-critical}.
\end{proof}

\begin{definition}
A Riemannian metric satisfying \eqref{eq:weak-Einstein-condition} is called a \emph{weakly $(2k)$-Einstein metric}. For $k=1$ this reduces to the notion of weakly Einstein metrics introduced by Euh, Park and Sekigawa \cite{EPS}.
\end{definition}

Thus, for $n>4k$, the critical points of the constrained functional $G_{2k}$ among $(2k)$-Thorpe (or anti-Thorpe) metrics are exactly the weakly $(2k)$-Einstein metrics.

\subsection{Hyper $(2k)$-Einstein metrics}

Recall that a Riemannian manifold $(M,g)$ is \emph{hyper $(2k)$-Einstein} (with $2k<n$) if
\[
\cc R^k = \lambda g^{2k-1}
\]
for some function $\lambda$. As noted in Section~5, such a metric is automatically both $(2k)$-Einstein and $(2k)$-Thorpe, and $\lambda$ is necessarily constant. The following identity, which is a special case of Theorem~\ref{thm:gen-Lanczos}, relates the two tensors appearing in the gradients of $G_{2k}$ and $H_{2k}$.

\begin{proposition}\label{prop:hyper-identity}
Let $n\geq 4k$ and suppose $(M,g)$ is hyper $(2k)$-Einstein with $\cc R^k=\lambda g^{2k-1}$. Then
\begin{equation}\label{eq:hyper-identity}
\frac{\cc^{4k-1}R^{2k}}{(4k-1)!}
   = \frac{\cc^{2k-1}(R^k\circ R^k)}{(2k-1)!}+\lambda^2 c(n,k)\,g,
\end{equation}
where $c(n,k)$ is a constant depending only on $n$ and $k$.
\end{proposition}

\begin{proof}
Under the hyper‑Einstein condition we have
\[
\cc^r R^k = \frac{(n-2k+r)!\,\lambda}{(2k-r)!\,(n-2k+1)!}\,g^{2k-r}
\qquad(1\leq r\leq 2k).
\]
Substituting these expressions into the general Lanczos identity \eqref{eq:gen-ident} (with $p=2k$) and simplifying yields \eqref{eq:hyper-identity}.
\end{proof}

Identity \eqref{eq:hyper-identity} shows that for a hyper $(2k)$-Einstein metric the tensor $\cc^{2k-1}(R^k\circ R^k)$ is proportional to $g$ if and only if the tensor $\cc^{4k-1}R^{2k}$ is proportional to $g$, i.e., if and only if the metric is $(4k)$-Einstein. Combining this observation with Proposition~\ref{prop:critical-Thorpe} and  Proposition~\ref{prop:hyper-identity} we obtain the following equivalence.

\begin{proposition}\label{prop:hyper-critical}
Let $n>4k$ and let $g$ be a hyper $(2k)$-Einstein metric. Then the following statements are equivalent:
\begin{enumerate}
\item $g$ is a critical point of $G_{2k}$ (restricted to unit volume);
\item $g$ is a weakly $(2k)$-Einstein metric;
\item $g$ is a $(4k)$-Einstein metric.
\end{enumerate}
\end{proposition}

This result illustrates the circle of ideas linking the critical points of the $L^2$-norm functional $G_{2k}$ with the various generalized Einstein conditions studied in the paper.

\end{document}